\documentclass[reqno]{amsart}
\usepackage{amssymb}
\usepackage{pdfsync}

\numberwithin{equation}{section}

\newtheorem{theorem}{Theorem}[section]

\newtheorem{lemma}[theorem]{Lemma}

\theoremstyle{definition}

\theoremstyle{definition} 
\newtheorem{remark}[theorem]{Remark}
\newtheorem{remarks}[theorem]{Remarks}

\newcommand{\bea}{\begin{eqnarray}}
\newcommand{\eea}{\end{eqnarray}}
\newcommand{\beas}{\begin{eqnarray*}}
\newcommand{\eeas}{\end{eqnarray*}}
\newcommand{\beq}{\begin{equation}}
\newcommand{\eeq}{\end{equation}}

\newcommand{\cA}{\mathcal A}
\newcommand{\cB}{\mathcal B}
\newcommand{\cC}{\mathcal C}
\newcommand{\cD}{\mathcal D}
\newcommand{\cE}{\mathcal E}
\newcommand{\cF}{\mathcal F}
\newcommand{\cG}{\mathcal G}
\newcommand{\cH}{\mathcal H}
\newcommand{\cR}{\mathcal R}

\newcommand{\cN}{\mathcal N}
\newcommand{\cMH}{\mathcal{MH}}
\newcommand{\cSM}{\mathcal{SM}}

\newcommand{\BB}{\mathbb B}

\newcommand{\R}{\mathbb R}

\newcommand{\EE}{\mathbb E}
\newcommand{\FF}{\mathbb F}

\newcommand{\ve}{\varepsilon}

\DeclareMathSymbol{\complement}{\mathord}{AMSa}{"7B}

\def\vv<#1>{\langle #1\rangle}
\def\Vv<#1>{\bigl\langle #1\bigr\rangle}

\begin{document}


\title[On Stefan problems with variable surface energy]
{On thermodynamically consistent Stefan problems with variable surface energy}

\author[J.~Pr\"uss]{Jan Pr\"uss}
\address{Institut f\"ur Mathematik \\
         Martin-Luther-Universit\"at Halle-Witten\-berg\\
         D-60120 Halle, Germany}
\email{jan.pruess@mathematik.uni-halle.de}

\author[G.~Simonett]{Gieri Simonett}
\address{Department of Mathematics\\
         Vanderbilt University \\
         Nashville, TN~37240, USA}
\email{gieri.simonett@vanderbilt.edu}

\author[M.~Wilke]{Mathias Wilke}
\address{Institut f\"ur Mathematik \\
         Martin-Luther-Universit\"at Halle-Witten\-berg\\
         D-60120 Halle, Germany}
\email{mathias.wilke@mathematik.uni-halle.de}



\begin{abstract}
A thermodynamically consistent two-phase Stefan problem with
temperature-dependent surface tension and with or without kinetic undercooling is studied.
It is shown that these problems generate local semiflows in well-defined state manifolds.
If a solution does not exhibit singularities, it is proved that it exists globally in time
and converges towards an equilibrium of the problem.
In addition, stability and instability of equilibria is studied.
In particular, it is shown that multiple spheres of the same radius are unstable if surface heat capacity is small;
however, if kinetic undercooling is absent, they are stable if surface heat capacity is sufficiently large.
\end{abstract}
\maketitle
\section{Introduction}

In the recent publication \cite{PSZ10} the authors studied Stefan problems with surface tension and with or without kinetic undercooling which are consistent with the laws of thermodynamics, in the sense that the total energy is preserved and the total entropy is strictly increasing along nonconstant smooth solutions.

\medskip

\noindent
{\bf 1.} \, To formulate this problem, let $\Omega\subset \R^{n}$ be a bounded domain of class $C^{2}$, $n\geq2$.
$\Omega$ is occupied by a material that can undergo phase changes: at time $t$, phase $i$ occupies
the subdomain $\Omega_i(t)$ of
$\Omega$, respectively, with $i=1,2.$
We assume that $\partial \Omega_1(t)\cap\partial \Omega=\emptyset$; this means
that no {\em boundary contact} can occur.
The closed compact hypersurface $\Gamma(t):=\partial \Omega_1(t)\subset \Omega$
forms the interface between the phases.

The problem consists in  finding a family of closed compact hypersurfaces $\Gamma(t)$ contained in $\Omega$
and an appropriately smooth function $u:\R_+\times\bar{\Omega}\rightarrow\R$ such that
\begin{equation}
\label{stefan}
\left\{\begin{aligned}
\kappa (u)\partial_t u-{\rm div}(d(u)\nabla u)&=0 &&\text{in}&&\Omega\setminus\Gamma(t)\\
\partial_{\nu} u &=0 &&\text{on}&&\partial \Omega \\
[\![u]\!]&=0 &&\text{on}&&\Gamma(t)\\
[\![\psi(u)]\!]+\sigma \cH &=\gamma(u) V &&\text{on}&&\Gamma(t) \\
[\![d(u)\partial_\nu u]\!] &=(l(u)-\gamma(u)V) V &&\text{on}&&\Gamma(t)\\
 u(0)&=u_0 &&\text{in} &&\Omega\setminus\Gamma_0,\quad\\
 \Gamma(0)&=\Gamma_0. && &&
\end{aligned}\right.
\end{equation}
Here
 $u(t)$ denotes the (absolute) temperature,
 $\nu(t)$ the outer normal field of $\Omega_1(t)$,
 $V(t)$ the normal velocity of $\Gamma(t)$,
 $\cH(t)=\cH(\Gamma(t))=-{\rm div}_{\Gamma(t)} \nu(t)$ the sum of the principal curvatures, 
 and
 $[\![v]\!]=v_2|_{\Gamma(t)}-v_1|_{\Gamma(t)}$ the jump of a (continuous) function $v$ across $\Gamma(t)$.
Since $u$ means absolute temperature we always assume that  $u>0$.

Several quantities are derived from the free energies $\psi_i(u)$ as follows:
\begin{itemize}
\item
 $\epsilon_i(u):= \psi_i(u)+u\eta_i(u)$ denotes the internal energy in phase $i$,
\item
 $\eta_i(u) :=-\psi_i^\prime(u)$ the entropy,
\item
 $\kappa_i(u):= \epsilon^\prime_i(u)=-u\psi_i^{\prime\prime}(u)$ the  heat capacity,
\item
$l(u):=u[\![\psi^\prime(u)]\!]=-u[\![\eta(u)]\!]$ the latent heat.
\end{itemize}
Furthermore, $d_i(u)>0$ denotes the coefficient of heat conduction in Fourier's law,
 $\gamma(u)\geq0$ the coefficient of  kinetic undercooling, and $\sigma>0$ the coefficient of surface tension.
 In the sequel we drop the index $i$, as there is no danger of confusion; we just keep in mind that the coefficients 
 in the bulk depend on the phases.

The temperature is assumed to be continuous across the interface. However, the free energy and the conductivities depend on the respective phases,
and hence the jumps $\varphi(u):=[\![\psi(u)]\!]$,
$[\![\kappa(u)]\!]$, $[\![\eta(u)]\!]$, $[\![d(u)]\!]$ are in general non-zero at the interface.
Throughout we require that the heat capacities $\kappa_i(u)$ and diffusivities $d_i(u)$ are strictly positive over the whole temperature range $u>0$, and that $\varphi$ has exactly one zero $u_m>0$ called the {\em melting temperature}.

If we assume that the coefficient of surface tension $\sigma$ is constant, then this model is consistent with the laws of thermodynamics. In fact,
the {\sf total energy} of the system is given by
\begin{equation}\label{energy}
{\sf E}(u,\Gamma) = \int_{\Omega\setminus\Gamma} \epsilon(u)\,dx +
\int_\Gamma \sigma\, ds,
 \end{equation}
 and by the transport and surface transport theorem we have for smooth solutions
 \begin{align*}
 \frac{d}{dt}{\sf E}(u(t),\Gamma(t)) &= -\int_\Gamma \{[\![d(u)\partial_\nu u]\!] +[\![\epsilon(u)]\!]V
  + \sigma \cH V\}\,ds\\
&= -\int_\Gamma \{[\![d(u)\partial_\nu u]\!] -(l(u)-\gamma(u)V))V\}\,ds=0,
 \end{align*}
and thus, energy is conserved. Also the {\sf total entropy} $\Phi(u,\Gamma)$ defined by
\begin{equation}\label{entropy}
\Phi(u,\Gamma) = \int_{\Omega\setminus\Gamma} \eta(u)\,dx
 \end{equation}
 is nondecreasing along smooth solutions, as
  \begin{align*}
 \frac{d}{dt}\Phi(u(t),\Gamma(t)) &=\int_\Omega\frac{1}{u^2} d(u)|\nabla u|^2\,dx
 - \int_\Gamma \frac{1}{u}\{ [\![d(u)\partial_\nu u]\!]+u[\![\eta(u)]\!]V\}\,ds\\
 &=\int_\Omega \frac{1}{u^2}d(u)|\nabla u|^2\,dx
 + \int_\Gamma \frac{1}{u}\gamma(u) V^2\,ds\ge 0.
\end{align*}

\medskip

\noindent
{\bf 2.} \, In this paper we consider the physically important case where surface tension $\sigma=\sigma(u)$ is a function of surface temperature $u$. 
We refer to \cite{Gur07, CLV01, DrPa99, NaSc03, NVC02} 
for background information on the importance of variable surface tension
in fluid flows and phase transitions.

Then, following \cite{Ish06} and \cite{Gur07},  
the surface energy will be $\int_\Gamma \epsilon_\Gamma(u)\,ds$ instead of $\int_\Gamma\sigma\,ds$, where $\epsilon_\Gamma(u)$ denotes the density of surface energy.
In addition, one has to
take into account the total surface entropy $\int_\Gamma \eta_\Gamma(u) \,ds$, as well as balance of surface energy.
The latter means that the Stefan law has to be replaced by a dynamic equation on the moving interface $\Gamma(t)$ of the form
\begin{equation*}
\kappa_\Gamma(u) \partial_{t,n} u - {\rm div}_\Gamma (d_\Gamma(u)\nabla_\Gamma u)
= [\![d(u)\partial_\nu u]\!] -\big(l(u)-\gamma(u)V+l_\Gamma(u)\cH\big)V,
\end{equation*}
where $\partial_{t,n}$ denotes the time derivative in normal direction,
see \eqref{normal-derivative}. As in the bulk we define on the interface
\begin{itemize}
\item
 $\epsilon_\Gamma(u):= \sigma(u)+u\eta_\Gamma(u)$, the surface internal energy,
\item
 $\eta_\Gamma(u) :=-\sigma^\prime(u)$, the surface entropy,
\item
 $\kappa_\Gamma(u):= \epsilon^\prime_\Gamma(u)=-u\sigma^{\prime\prime}(u)$, the  surface heat capacity,
\item
$l_\Gamma(u):=u\sigma^\prime(u)=-u\eta_\Gamma(u)$, the surface latent heat.
\end{itemize}
We also employ Fourier's law on the interface to describe surface heat conduction, i.e. we set
$q_\Gamma :=-d_\Gamma(u)\nabla_\Gamma u$, which should be present as soon as the interface has heat capacity. Recalling that $u$ is assumed to be continuous across the interface the surface temperature
\begin{equation}
\label{u-Gamma}
u_\Gamma:=u_{|_\Gamma}
\end{equation}
is well-defined.

Obviously, if $\sigma$ is constant then $\epsilon_\Gamma=\sigma$, and $\eta_\Gamma=\kappa_\Gamma =l_\Gamma=0$, hence this model reduces to \eqref{stefan}. On the other hand, if $\sigma$ is linear in $u$ we still have $\kappa_\Gamma=0$ and then it makes sense to also set $d_\Gamma\equiv0$, to obtain the modified Stefan law
\begin{equation*}
 [\![d(u)\partial_\nu u]\!] =\big(l(u)-\gamma(u)V+l_\Gamma(u)\cH\big)V,
\end{equation*}
which differs from the Stefan law in \eqref{stefan} only by replacing $l(u)$ by $l(u)+l_\Gamma(u)\cH$. This is just a minor modification of \eqref{stefan}, and its analysis remains essentially the same as in \cite{PSZ10}. The only difference is that the stability condition for the equilibria, and in case $\gamma\equiv0$ also the well-posedness condition, changes.
More precisely, the well-posedness condition changes from $\varphi^\prime\neq0$ to $\lambda^\prime\neq0$ where $\lambda(s):=\varphi(s)/\sigma(s)$, and the stability condition modifies by replacing $\varphi^\prime/\sigma$ by $\lambda^\prime$.

Therefore we concentrate here on the  case where $\kappa_\Gamma(u),d_\Gamma(u)>0$, which means that $\sigma$ is strictly concave. It has been shown experimentally that positive surface heat capacity  $\kappa_\Gamma$
(as opposed to vanishing surface heat capacity) is important in
certain  practical situations; see \cite{Ward} for recent work in this direction.
Experimental evidence
also shows that $\sigma$ is  strictly decreasing, hence admits exactly one zero $u_c>0$; $\sigma(u)$ is positive in $(0,u_c)$ and negative for $u>u_c$. Physically, it is reasonable to assume $u_c>u_m$.
It turns out that the analysis of the problem with nonlinear surface tension is considerably different from the linear case.
In the sequel we always assume that
\begin{equation}
\label{regularity-coeff}
\begin{split}
 d_i,\psi_i,d_\Gamma,\sigma,\gamma\in C^3(0,u_c),\quad
 d_i,\kappa_i,d_\Gamma,\kappa_\Gamma,\sigma>0\;\;\text{on}\;\; (0,u_c),\quad i=1,2,
\end{split}
\end{equation}
if not stated otherwise.
Furthermore, we let
$\gamma\equiv0$ if there is no  undercooling, or $\gamma>0$ on $(0,u_c)$ if undercooling is present, and we 
restrict our attention to the temperature
range $u\in(0,u_c)$.

With these restrictions on the parameter functions, we consider the following problem:
\begin{equation}
\label{stefan-var}
\left\{\begin{aligned}
\kappa (u)\partial_t u-{\rm div}(d(u)\nabla u)&=0 &&\text{in}&&\Omega\setminus\Gamma(t)\\
\partial_{\nu} u &=0 &&\text{on}&&\partial \Omega \\
[\![u]\!]=0,\quad u_\Gamma&=u &&\text{on}&&\Gamma(t)\\
\varphi(u_\Gamma)+\sigma(u_\Gamma) \cH &=\gamma(u_\Gamma) V &&\text{on}&&\Gamma(t) \\
\kappa_\Gamma(u_\Gamma) \partial_{t,n} u_\Gamma - {\rm div}_\Gamma (d_\Gamma(u_\Gamma)\nabla_\Gamma u_\Gamma)&=&&\\
=[\![d(u)\partial_\nu u]\!] -(l(u_\Gamma)+l_\Gamma(u_\Gamma)\cH&-\gamma(u_\Gamma)V) V &&\text{on}&&\Gamma(t)\\
 u(0)&=u_0 && \text{in} && \Omega\setminus\Gamma_0,\\
 \Gamma(0)&=\Gamma_0. &&
\end{aligned}\right.
\end{equation}
 Here $\varphi(u)=[\![\psi(u)]\!]$, and $\partial_{t,n}u_\Gamma$
denotes the time derivative of $u_\Gamma$ in normal direction, defined by
\begin{equation}
\label{normal-derivative}
 \partial_{t,n}u_\Gamma(t,p):=\frac{d}{d\tau}u_\Gamma(t+\tau,x(t+\tau,p))\big|_{\tau=0}\,,
  \quad t>0,\quad p\in\Gamma(t),
\end{equation}
with
$\{x(t+\tau,p)\in\R^{n}:(\tau,p)\in (-\varepsilon,\varepsilon)\times\Gamma(t)\}$
the flow induced by the normal vector field $(V\nu)$. That is, $[\tau\mapsto x(t+\tau,p)]$ defines for each $p\in\Gamma(t)$
a flow line through $p$ with 
\begin{equation*}
\label{flow-line}
  \quad \frac{d}{d\tau}x(t+\tau,p)=(V\nu)(t+\tau,x(t+\tau,p)),
  \quad x(t+\tau,p)\in\Gamma(t+\tau),
  \quad \tau\in (-\varepsilon,\varepsilon),
\end{equation*}
and $x(t,p)=p$.
The existence of a unique trajectory
$$\{x(t+\tau,p)\in\R^{n}: \tau\in (-\varepsilon,\varepsilon)\},\quad p\in\Gamma(t),$$
with the above properties is not completely obvious,
see for instance \cite{MS99b} for a proof.

We note that the (non-degenerate) equilibria for this problem are the same as those for \eqref{stefan}: the temperature is constant, and the disperse phase $\Omega_1$ consists of finitely many nonintersecting balls of the same radius. We shall prove that such an equilibrium is stable in the state manifold $\cSM$ defined below if $\Omega_1$ is connected and the stability condition introduced in the next section holds. Such an equilibrium will be a local maximum of the total entropy, as we found before in \cite{PSZ10} for the case of constant surface tension.
To the best of our knowledge, there is no mathematical work on thermodynamically consistent
Stefan problems with surface tension depending on the temperature.

\medskip

\noindent
{\bf 3.} \,  The case where undercooling is present is the simpler one, as both equations on the interface are dynamic equations. 
In particular, the Gibbs-Thomson identity
$$ \gamma(u_\Gamma)V-\sigma(u_\Gamma)\cH= \varphi(u_\Gamma)$$
can be understood as a {\bf mean curvature flow} for the evolution of the surface, modified by physics.

If there is no undercooling, it is convenient to eliminate the time derivative of $u_\Gamma$ from the energy balance on the interface. In fact, differentiating the Gibbs-Thomson law w.r.t.\ time $t$ and, with $\lambda(s)=\varphi(s)/\sigma(s)$, we obtain
$$ \lambda^\prime(u_\Gamma)\partial_{t,n}u_\Gamma +\cH^\prime(\Gamma) V=0\quad \mbox{ on } \; \Gamma(t),$$
where $\cH^\prime(\Gamma)={\rm tr}\, L_\Gamma^2 +\Delta_\Gamma$,
with $L_\Gamma$ the Weingarten tensor and $\Delta_\Gamma$ the Laplace-Beltrami operator of $\Gamma$.
(These quantities will be introduced in Section 3).
Hence substitution into surface energy balance yields with
\begin{equation*}
\label{T-Gamma}
T_\Gamma(u_\Gamma) :=\omega_\Gamma(u_\Gamma)-\cH^\prime(\Gamma),
\;
\omega_\Gamma (u_\Gamma):=\lambda^\prime(u_\Gamma)(l(u_\Gamma)-
l_\Gamma(u_\Gamma)\lambda(u_\Gamma))/\kappa_\Gamma(u_\Gamma),
\end{equation*}
the relation
\begin{equation}
\label{T-Gamma-V}
T_\Gamma(u_\Gamma) V=\frac{\lambda^\prime(u_\Gamma)}{\kappa_\Gamma(u_\Gamma)}
\big\{{\rm div}_\Gamma (d_\Gamma(u_\Gamma)\nabla_\Gamma u_\Gamma)+ [\![d(u)\partial_\nu u]\!]\big\}.
\end{equation}
As $V$ should be determined only by the state of the system and should not depend on time derivatives of other variables,
this indicates that the problem without undercooling is not well-posed if the operator $T_\Gamma(u_\Gamma)$
is not invertible in $L_2(\Gamma)$, as $V$ might not be well-defined. On the other hand if $T_\Gamma(u_\Gamma)$ is invertible, then
\begin{equation}
V = [T_\Gamma(u_\Gamma)]^{-1}\frac{\lambda^\prime(u_\Gamma)}{\kappa_\Gamma(u_\Gamma)}
\Big\{{\rm div}_\Gamma (d_\Gamma(u_\Gamma)\nabla_\Gamma u_\Gamma)+
[\![d(u)\partial_\nu u]\!]\Big\}
\end{equation}
uniquely determines the interfacial velocity $V,$ gaining two derivatives in space, and showing that the right hand side of surface energy balance is of lower order. Note that
\begin{equation}
\omega_\Gamma(s) = s\sigma(s)[\lambda^\prime(s)]^2/\kappa_\Gamma(s)\geq 0 \quad \mbox{ in } (0,u_c),
\end{equation}
and $\omega_\Gamma(s)=0$ if and only if $\lambda^\prime(s)=0$. Therefore the well-posedness condition becomes more complex compared to the case $\kappa_\Gamma\equiv0$.

Going one step further, taking the surface gradient of the Gibbs-Thomson relation yields the identity
\begin{equation}
\label{mcflow}
\kappa_\Gamma(u_\Gamma) V - d_\Gamma(u_\Gamma)\cH(\Gamma)
= \kappa_\Gamma(u_\Gamma)\{ f_\Gamma(u_\Gamma)+F_\Gamma(u,u_\Gamma)\},
\end{equation}
as will be shown in Section~6. Here the function $f_\Gamma$ is the antiderivative of $\lambda(d_\Gamma/\kappa_\Gamma)^\prime$ vanishing at $s=u_m$, and
 $F_\Gamma$ is nonlocal in space and of  lower order. So also in the case where undercooling is absent we obtain a
 {\bf mean curvature flow}, modified by physics.
 
Here some remarks about the nature of $T_\Gamma$ are in order.
$T_\Gamma$ is a mathematical quantity which does not seem to allow for a physical interpretation.
In case that $\Gamma$ coincides with an equilibrium of  system \eqref{stefan-var},
invertibility of $T_\Gamma$ is characterized by the conditions 
$l_*\neq 0$ and $\eta_*\neq 1$, where $l_*$ and $\eta_*$ are defined below.
As $T_\Gamma$ contains the term $\Delta_\Gamma$, a second order differential operator
acting on functions defined on $\Gamma$, $T^{-1}_\Gamma$ (and hence also $V$) will gain two 'spacial' derivatives. 
\goodbreak
\medskip

\noindent
{\bf 4.} 
Since we do not impose any structural assumptions on the free energy,
the diffusivity, and the surface tension at $\theta=0$,
is is not possible to show that the temperature $\theta(t)$ remains positive.
It would be an important question to characterize constitutive laws which ensure this property.

On the other side, we can also not ensure that solutions stay bounded away from $u_c$.
Note that the model is not meaningful for $u>u_c$, as the phases are then no longer separated.
This region would correspond to a plasma.

In our model we do not allow for the interface $\Gamma$ to touch the boundary of $\Omega$.
We refer to \cite{BoPr15} for modeling aspects concerning this situation.

\medskip

The plan for this paper is as follows. In Section~2 we discuss some fundamental physical properties of the Stefan problem with variable surface tension. In particular, it is shown that the negative total entropy is a strict Lyapunov functional for the problem, and we characterize and analyze the equiliria of the system. The direct mapping method based on the 
Hanzawa transform, first introduced in \cite{Han81}, is discussed in Section~3. This way the problem is reduced to a quasilinear parabolic problem. In Section~4 we consider the full linearization of the problem at a given equilibrium, and we prove that these are normally hyperbolic, generically. The last two sections deal with the analysis of the nonlinear problem with and without kinetic undercooling. The analysis is based on results for abstract quasilinear parabolic problems, in particular on the generalized principle of linearized stability, see 
\cite{KPW10,PSZ09}. We refer here to \cite{DHP03,Mey10,PrSi04} for information on maximal 
regularity in $L_p$- and weighted $L_p$- spaces,
and to \cite{EPS03,Gu86,Gu88,PSZ10} for more background information concerning the Stefan problem.

\section{Energy, Entropy and Equilibria}
\noindent
{\bf(a)}\,  The {\sf total energy} of the system  \eqref{stefan-var} is given by
\begin{equation}\label{energy-var}
{\sf E}(u,\Gamma) = \int_{\Omega\setminus\Gamma} \epsilon(u)\,dx +
\int_\Gamma \epsilon_\Gamma(u_\Gamma)\,ds,
 \end{equation}
 and by the transport and surface transport theorem we have for smooth solutions
 \begin{align*}
 \frac{d}{dt}{\sf E}(u,\Gamma) &= \int_\Omega \kappa(u)\partial_tu \,dx-\int_\Gamma [\![\epsilon(u)]\!]V\,ds
  + \int_\Gamma\{\kappa_\Gamma(u_\Gamma)\partial_{t,n} u_\Gamma -\epsilon_\Gamma(u_\Gamma) \cH V\}\,ds\\
&= \int_\Gamma \{-[\![d(u)\partial_\nu u]\!]-[\![\epsilon(u)]\!]V +{\rm div}_\Gamma( d_\Gamma(u_\Gamma)\nabla_\Gamma u_\Gamma)\\
&\qquad\;\; +[\![d(u)\partial_\nu u]\!]-(l(u)+l_\Gamma(u_\Gamma)\cH)V+\gamma(u_\Gamma)V^2-\epsilon_\Gamma(u_\Gamma)\cH V\}\,ds\\
&= -\int_\Gamma\{ [\![\psi(u)]\!]+\sigma(u_\Gamma)\cH -\gamma(u_\Gamma)V\}V\,ds=0
 \end{align*}
by the Gibbs-Thomson law, and thus, energy is conserved.
\medskip\\
\noindent
{\bf(b)}\, The {\sf total entropy} of the system, given by
 \begin{equation}\label{entropy-st}
 \Phi(u,\Gamma)= \int_{\Omega\setminus\Gamma} \eta(u)\,dx +\int_\Gamma \eta_\Gamma(u_\Gamma)\,ds,
 \end{equation}
 satisfies
 \begin{align*}
 \frac{d}{dt}\Phi(u,\Gamma)&= \int_\Omega \eta^\prime(u)\partial_t u\,dx +\int_\Gamma\{\partial_{t,n}\eta_\Gamma(u_\Gamma) -([\![\eta(u)]\!]+\eta_\Gamma(u_\Gamma)\cH)V\}\,ds\\
 &= \int_\Omega \frac{1}{u}\kappa(u)\partial_tu\,dx
 +\int_\Gamma\frac{1}{u_\Gamma}\{\kappa_\Gamma(u_\Gamma)\partial_{t,n}u_\Gamma
 +(l(u)+l_\Gamma(u_\Gamma)\cH)V\}\,ds\\
 &=\int_\Omega\frac{1}{u^2} d(u)|\nabla u|^2\,dx \\
 &+ \int_\Gamma\frac{1}{u_\Gamma}
 \{ -[\![d(u)\partial_\nu u]\!]+{\rm div}_\Gamma (d_\Gamma(u_\Gamma)\nabla_\Gamma u_\Gamma)
 +[\![d(u)\partial_\nu u]\!]+\gamma(u_\Gamma)V^2\}\,ds\\
 &=\int_\Omega \frac{1}{u^2}d(u)|\nabla u|^2\, dx + \int_\Gamma\frac{1}{u_\Gamma^2}\{d_\Gamma(u_\Gamma)|\nabla_\Gamma u_\Gamma|^2+ u_\Gamma\gamma(u_\Gamma) V^2\}\,ds\ge 0,
 \end{align*}
 where we employed the transport theorem, the surface transport theorem and \eqref{stefan-var}.
In particular, the negative total entropy is a Lyapunov functional
for problem \eqref{stefan-var}.

\bigskip

\noindent
{\bf(c)}\, Even more,
$-\Phi$ is a strict Lyapunov functional in the sense that it is strictly decreasing along
 smooth solutions which are non-constant in time.
Indeed, if at some time
$t_0\geq0$ we have
 \begin{equation*}
 \frac{d}{dt}\Phi(u(t_0),\Gamma(t_0)) =0,
 \end{equation*}
 then
 \begin{equation*}
 \int_\Omega \frac{1}{u^2}d(u)|\nabla u|^2\, dx+ \int_\Gamma\frac{1}{u_\Gamma^2}d(u)|\nabla u_\Gamma|^2\, ds
 + \int_\Gamma \frac{1}{u_\Gamma}\gamma(u_\Gamma) V^2\,ds = 0,
 \end{equation*}
hence $\nabla u(t_0)=0$
in $\Omega$  and $\gamma(u_\Gamma(t_0))V(t_0)=0$ on $\Gamma(t_0)$. This implies $u(t_0)=const=u_\Gamma(t_0)$ in $\Omega$, and $\cH(t_0)=-[\![\psi(u(t_0))]\!]/\sigma(u_\Gamma(t_0))=const$,
provided we have
\begin{equation}\label{sigma}
[\![\psi(s)]\!]=0\quad \Rightarrow \quad \sigma(s)>0.
\end{equation}
Physically, this assumption is plausible, as it means that at melting temperature $u_m>0$ (defined as the {\em unique} positive zero of  the function $\varphi(s):=[\![\psi(s)]\!]$) the surface tension $\sigma(u_m)$ is positive.
 Since $\Omega$ is bounded, we may conclude that $\Gamma(t_0)$ is a union of finitely many, say $m$, disjoint spheres of equal radius,
i.e.\ $(u(t_0),\Gamma(t_0))$ is an equilibrium.
Therefore, the {\em $\omega$ limit set} of solutions 
{\em within}
the state manifold defined below are contained in the $(mn+1)$-dimensional manifold of equilibria
\begin{equation}
\label{equilibria-I}
\begin{aligned}
\cE&=\big\{\big(u_\ast,\bigcup_{1\le l\le m}S_{R_\ast}(x_l)\big):
0<u_\ast<u_c,\ [\![\psi(u_\ast)]\!]=(n-1)\sigma(u_*)/R_*,\\
&\hspace{4cm}\bar B_{R_\ast}(x_l)\subset \Omega,\, \bar{B}_{R_\ast}(x_l)\cap \bar{B}_{R_\ast}(x_k)=\emptyset,\
 k\neq l\big\},
\end{aligned}
\end{equation}
where $S_{R_\ast}(x_l)$ denotes the sphere with radius $R_\ast$ and center $x_l$.

\medskip

\noindent
{\bf(d)}\, Another interesting observation is the following. Consider the
critical points of the functional $\Phi(u,u_\Gamma,\Gamma)$ with constraint
${\sf E}(u,u_\Gamma,\Gamma)={\sf E}_0$, say on
$$U:=\{(u,u_\Gamma,\Gamma):\; u\in C(\bar{\Omega}\setminus\Gamma),\;\;
\Gamma\in \cMH^2(\Omega),\;\; u_\Gamma\in C(\Gamma),\;\; u,u_\Gamma>0\},$$
 see below for the definition of
$\cMH^2(\Omega)$. So here we do not assume from the beginning that $u$ is continuous across $\Gamma$,
and $u_\Gamma$ denotes surface temperature.
Then by the method of Lagrange multipliers, there
is $\mu\in\R$ such that at a critical point $(u_*,u_{\Gamma*},\Gamma_*)$ we have
\begin{equation}
\label{VarEq} \Phi^\prime(u_*,u_{\Gamma*},\Gamma_*)+\mu {\sf E}^\prime(u_*,u_{\Gamma*},\Gamma_*)=0.
\end{equation}
The derivatives of the functionals are given by
\begin{equation*}
\langle \Phi^\prime(u,u_\Gamma,\Gamma) |(v,v_\Gamma,h)\rangle
= (\eta^\prime(u)|v)_\Omega + (\eta^\prime_\Gamma(u_\Gamma)|v_\Gamma)_\Gamma -([\![\eta(u)]\!]+\eta_\Gamma(u_\Gamma)\cH(\Gamma)|h)_{\Gamma},
\end{equation*}
and
$$\langle {\sf E}^\prime(u,u_\Gamma,\Gamma) |(v,v_\Gamma,h)\rangle = (\epsilon^\prime(u)|v)_\Omega
+(\epsilon^\prime_\Gamma(u_\Gamma)|v_\Gamma)_\Gamma -([\![\epsilon(u)]\!]+\epsilon_\Gamma(u_\Gamma) \cH(\Gamma)|h)_{\Gamma}.$$
Setting first $v_\Gamma=h=0$ and varying $v$ in \eqref{VarEq} we obtain
$$\eta^\prime(u_*) + \mu \epsilon^\prime(u_*)=0\quad \mbox{ in } \Omega,$$
 varying $v_\Gamma$ yields
$$\eta_\Gamma^\prime(u_{\Gamma*}) +\mu \epsilon_\Gamma^\prime(u_{\Gamma*})=0\text{ on $\Gamma_*$},$$
and finally
varying $h$ we get
$$[\![\eta(u_*)]\!] +\eta_\Gamma(u_{\Gamma*})\cH(\Gamma)+\mu([\![\epsilon(u_*)]\!]+\epsilon_\Gamma(u_{\Gamma*}) \cH(\Gamma_*))=0\text{ on $\Gamma_*$}.$$
The relations $\eta(u)=-\psi^\prime(u)$ and
$\epsilon(u)=\psi(u)-u\psi^\prime(u)$ imply
$0=-\psi^{\prime\prime}(u_*)(1+\mu u_*)$, and this shows that
$u_*=-1/\mu$ is constant in $\Omega$, since
$\kappa(u)=-u\psi^{\prime\prime}(u)>0$ for all $u>0$ by assumption.
Similarly on $\Gamma_*$ we obtain that $u_{\Gamma*}=-1/\mu$ is constant as well, 
provided $\kappa_\Gamma(u_\Gamma)>0$, hence in particular $u_*\equiv u_{\Gamma*}$.
This further implies the Gibbs-Thomson relation $[\![\psi(u_*)]\!]+\sigma(u_*) \cH(\Gamma_*)=0$.
 Since $u_*$ is constant we see
that $\cH(\Gamma_*)$ is constant, by \eqref{sigma}.
Therefore $\Gamma_*$ is a sphere
whenever connected, and a union of finitely many disjoint spheres of
equal size otherwise. Thus the critical points of the entropy
functional for prescribed energy are precisely the equilibria of
problem \eqref{stefan-var}.

\bigskip

\noindent
{\bf(e)}\, Going further, suppose we have an equilibrium $e_*:=(u_*,u_{\Gamma*},\Gamma_*)$
where the total entropy has a local maximum w.r.t.\ the constraint
${\sf E}={\sf E}_0$ constant.
Then $\cD_*:=[\Phi+\mu {\sf E}]^{\prime\prime}(e_*)$ is negative semi-definite on the kernel of
${\sf E}^\prime(e_*)$, where $\mu=-1/u_*$
 is the fixed Lagrange
multiplier found above. The kernel of ${\sf E}^\prime(e)$ is given
by the identity
\begin{equation*}
(\kappa(u)| v)_\Omega + (\kappa_\Gamma(u_\Gamma)|v_\Gamma)_\Gamma-([\![\epsilon(u)]\!] + \epsilon_\Gamma(u_\Gamma) \cH(\Gamma)| h)_\Gamma =0,
\end{equation*}
which at equilibrium yields
\begin{equation}
\label{kE}
(\kappa_\ast| v)_\Omega + (\kappa_{\Gamma*}|v_\Gamma)_\Gamma+ u_*(l_\ast| h)_\Gamma=0,
\end{equation}
where $\kappa_\ast :=\kappa(u_\ast)$, $\kappa_{\Gamma*}:=\kappa_{\Gamma}(u_*)$ and
\begin{equation}
\label{l-stern}
l_*:=\frac{1}{u_*}\big\{l(u_*)+l_{\Gamma}(u_*)\cH(\Gamma_*)\big\}
=[\![\psi^\prime(u_*)]\!]+\sigma^\prime(u_*)\cH(\Gamma_*). 
\end{equation}
On the other hand, a straightforward calculation yields with
$z=(v,v_\Gamma,h)$
\begin{align}
\label{2var}
-\langle \cD_* z|z\rangle &=\frac{1}{u_\ast^2}\big[ (\kappa_\ast v|v)_\Omega + (\kappa_{\Gamma*}v_\Gamma|v_\Gamma)_\Gamma
- \sigma_* u_*( \cH^\prime(\Gamma_*) h|h)_\Gamma\big],
\end{align}
where $\kappa_{\Gamma*}=\kappa_{\Gamma}(u_*)$ and $\sigma_*=\sigma(u_*)$.
As  $\kappa_\ast$ and $\kappa_{\Gamma*}$ are positive, we see that the form
$\langle \cD z|z\rangle$ is negative semi-definite  as soon as
$\cH^\prime ({\Gamma_*}) $ is negative semi-definite.
 We have
\begin{equation*}
\cH^\prime(\Gamma_\ast) = (n-1)/{R^2_\ast} + \Delta_\ast,
\end{equation*}
where $\Delta_\ast$ denotes the
Laplace-Beltrami operator on $\Gamma_\ast$, and $R_\ast$ means the radius of
an equilibrium sphere. To derive necessary conditions for an equilibrium $e_*$ to be a
local maximum of entropy, we consider two cases.
\smallskip\\
\noindent {1.} Suppose that $\Gamma_\ast$ is not connected,
i.e. $\Gamma_\ast$ is a finite union of spheres $\Gamma^k_\ast$.
Set $v=v_\Gamma=0$, and let $h=h_k$ be
constant on $\Gamma^k_\ast$
with $\sum_k h_k=0$.
Then the constraint \eqref{kE} holds,
and with $\omega_n$ the surface area of the unit sphere in $\R^n$
\begin{equation*}
\langle \cD z|z\rangle= (\sigma_*u_\ast)((n-1)/R^2_\ast)\omega_n R^{n-1}_* \,\sum_{k=1}^m h_k^2 >0,\
\end{equation*}
hence $\cD$ cannot be negative semi-definite in this case, as $\sigma_*>0$ by \eqref{sigma}. Thus if $e_\ast$
is an equilibrium with maximal total entropy, then $\Gamma_\ast$ must be
connected, and hence both phases are connected.
\smallskip\\
\noindent 2. Assume that $\Gamma_\ast$ is connected. With
$h= -(\kappa_*|1)_\Omega-\kappa_{\Gamma*}|\Gamma_*|$,
$v=v_\Gamma=u_*l_*|\Gamma_*|$
 we see that
 $\cD$ negative semi-definite on the kernel of ${\sf E}^\prime(e_\ast)$ implies
the condition
\begin{equation}
\label{var-sc}
\zeta_\ast:=\zeta(u_*):=\frac{(n-1)\sigma_* [(\kappa_\ast|1)_\Omega+\kappa_{\Gamma*}|\Gamma_*|]}{u_*l^2_\ast R^2_\ast |\Gamma_\ast|}\le 1.
\end{equation}
We will see below that connectedness of $\Gamma_*$ and the {\em strong stability condition} $\zeta_*<1$ are sufficient for stability of the equilibrium $e_*$.

We point out that the quantity $\zeta_\ast$ defined in \eqref{var-sc}
coincides with the analog quantity in \cite[Defintion (1.11)]{PSZ10}
in case $\kappa_{\Gamma*}=0$ and $\sigma=$ constant.
(Note that $l_\ast=l(u_*)/u_*$ in this case, which differs from the definition of $l_\ast$
in \cite{PSZ10}).
\medskip

\noindent
{\bf(f)}\, Summarizing, we have shown
\begin{itemize}
\item The total energy is constant along smooth solutions of \eqref{stefan-var}.
\item The negative total entropy is a strict Ljapunov functional for \eqref{stefan-var}.
\item The equilibria of \eqref{stefan-var} are precisely the critical points
of the entropy functional with prescribed energy.
\item
If the entropy functional with prescribed energy has a local maximum at $e_*=(u_*,u_{\Gamma*},\Gamma_*)$
then $\Gamma_\ast$ is connected.
\item
If $\Gamma_\ast$ is connected, a necessary condition
for a critical point $(u_\ast,u_{\Gamma*},\Gamma_\ast)$ to be a local maximum of
the entropy functional with prescribed energy is inequality \eqref{var-sc}.
\end{itemize}
\medskip

\noindent
{\bf (g)}\,
We would like to point out a phenomenon, in the absence of kinetic undercooling, which is to positive surface heat capacity $\kappa_\Gamma$. If $\kappa_\Gamma$ at an equilibrium 
$(u_\ast,u_{\Gamma*},\Gamma_\ast)$ is large enough and $\Gamma_*$ is disconnected,
then such a steady state is stable, see Theorem~\ref{lin-stability} and 
Theorem~\ref{nonlinstab0}.
Hence, this case seems to prevent the onset of Ostwald ripening.
However, such equilibria cannot be maxima of the total entropy. 

	This is in strict contrast to the situation where the surface tension $\sigma$ is constant.
	In this case it is shown in \cite{PSZ10} that
	multiple spheres (of the same radius) are always unstable for \eqref{stefan}.
	This situation is reminiscent of the onset of {\sl Ostwald ripening},
	a process that manifests itself in the way that larger structures grow 
	while smaller ones shrink and disappear.
	Here we refer to  
	\cite{AlFu03, AlFuKa03,AlFuKa04,AlFuKa04b,GORS09,HNO05},
	\cite{Niet99}-\cite{NietVe04} and the references therein
	for various aspects and results on Ostwald ripening.
	In particular, we mention that the authors in 
	 \cite{AlFu03, AlFuKa03, AlFuKa04,AlFuKa04b} use the quasi-stationary Stefan problem
	with surface tension (i.e., the Mullins-Sekerka problem) to model Ostwald ripening.
	Under proper scaling assumptions, the way
	sphere-like particles evolve is analyzed.
	Interesting and illuminating connections between various versions 
	of the Stefan problem (mostly the Mullins-Sekerka problem)
	and Ostwald ripening are given in 
	\cite{HNO05,Niet99,Niet00,Niet01b,NietVe04}.
	It would be of considerable interest to also pursue the effect of coarsening in
	the framework of the thermodynamically consistent Stefan problem \eqref{stefan} and 
	\eqref{stefan-var}

	In case $\Gamma_*$ is connected we show that an equilibrium  
	 is stable if $\zeta_*<1$, and unstable 
	if $\zeta_*>1$. This situation is in accordance with the results found in \cite{PSZ10}
	for the case of constant surface tension. 
	Here we mention that stability of a connected equilibrium for the Stefan problem with constant surface tension 
	has also been obtained in \cite{Ha12}. 
 We refer to the introduction of \cite{PSZ10} for a detailed discussion of the literature. 	 
\medskip

\noindent
{\bf(h)}\, Now let us look at the energy of an equilibrium as a function of temperature. Suppose we have an equilibrium $(u,\Gamma)$ at a given energy level ${\sf E_0}$, and assume that $\Gamma$ consists of $m$ disjoint spheres of radius $R$ contained in $\Omega$. Then
$$0<R<R_m:=\sup\{ R>0: \, \Omega \mbox{ contains $m$ disjoint ball of radius } R\},$$
and with $\varphi(u):=[\![\psi(u)]\!]$ we have
$$ 0=\varphi(u)+\sigma(u)\cH(\Gamma)= \varphi(u) -(n-1)\sigma(u)/R,$$
and hence
$ R=R(u)= (n-1)\sigma(u)/\varphi(u).$
Further we have
\begin{align*}
{\sf E}_e(u)&:= {\sf E}(u,\Gamma)= \int_\Gamma \epsilon(u)\,dx +\int_\Gamma \epsilon_\Gamma(u)\,ds\\
&= \epsilon_2(u)|\Omega|- |\Omega_1|[\![\epsilon(u)]\!] + \epsilon_\Gamma(u)|\Gamma|\\
&= \epsilon_2(u)|\Omega| - (m\omega_n/n) R(u)^n[\![\epsilon(u)]\!]+m\omega_n R(u)^{n-1}\epsilon_\Gamma(u)\\
& =\epsilon_2(u)|\Omega|+c_{n,m}\Big[\frac{\sigma(u)^n}{\varphi(u)^{n-1}} - u\frac{d}{du}\frac{\sigma(u)^n}{\varphi(u)^{n-1}}\Big],
\end{align*}
where $c_{n,m}= m\omega_n(n-1)^{n-1}/n.$ Thus we obtain for the total energy of an equilibrium
\begin{equation}\label{eq-energy}
{\sf E}_e(u)= \delta(u)-u\delta^\prime(u),\quad \delta(u) =|\Omega|\psi_2(u)+c_{n,m}\frac{\sigma(u)^n}{\varphi(u)^{n-1}}.
\end{equation}
Consequently, the equilibrium temperature for an equilibrium, where $\Gamma$ consists of $m$ components, is the solution of the scalar problem
$${\sf E}_0 = {\sf E}_e(u) = \delta(u)-u\delta^\prime(u),\quad 0<u<u_c,
\quad 0<\sigma(u)/\varphi(u)<R_m/(n-1).$$
Let us look at the derivative of the function ${\sf E}_e(u)$. A simple calculation yields
\begin{equation*}
\begin{aligned}
{\sf E}^\prime_e(u)&=-u\delta^{\prime\prime}(u)=-|\Omega|u\psi_2^{\prime\prime}(u)
-c_{n,m}u\frac{d}{du}\big[n\left(\sigma/\varphi\right)^{n-1}\sigma^\prime
-(n-1)(\sigma/\varphi)^n\varphi^\prime\big]\\
&= |\Omega|\kappa_2(u)-c_{n,m}u\big[n(\sigma/\varphi)^{n-1}\sigma^{\prime\prime}-(n-1)(\sigma/\varphi)^n\varphi^{\prime\prime}\big]\\
&\quad -c_{n,m}n(n-1)(\sigma/\varphi)^{n-1}u\big[(\sigma^\prime)^2/\sigma -
2\sigma^\prime \varphi^\prime/\varphi+ \sigma (\varphi^\prime)^2/\varphi^2\big]\\
&= |\Omega|\kappa_2(u)+ |\Gamma|\kappa_\Gamma(u) -|\Omega_1|[\![\kappa(u)]\!]
- ({R^2|\Gamma|}/{(n\!-\!1)\sigma})u
 \big[ \varphi^\prime - \sigma^\prime \varphi/\sigma\big]^2\\
&=\big[(\kappa(u)|1)_\Omega+|\Gamma|\kappa_\Gamma(u)\big]
\!-\!
{(R^2(u)|\Gamma|}/{(n\!-\!1)\sigma(u)}) u\big[\,[\![\psi^\prime(u)]\!]+\sigma^\prime(u)\cH(\Gamma)\big]^2.
\end{aligned}
\end{equation*}
Therefore the stability condition $\zeta(u)\leq 1$ is equivalent to ${\sf E}^\prime_e(u)\leq0$, an alternative interpretation to the one obtained above.
\section{Transformation to a Fixed Interface}
Let $\Omega\subset\R^n$ be a bounded domain with boundary $\partial
\Omega$ of class $C^2$, and suppose $\Gamma\subset \Omega$ is a
closed hypersurface of class $C^2$, i.e.\ a $C^2$-manifold which is
the boundary of a bounded domain $\Omega_1\subset \Omega$. We then
set $\Omega_2=\Omega\setminus\bar{\Omega}_1$. Note that while
$\Omega_2$ typically is connected, $\Omega_1$ may be  disconnected. However,
$\Omega_1$ consists of finitely many components only, as $\partial
\Omega_1=\Gamma$ by assumption is a manifold, at least  of class
$C^2$. 
In the following, 
we refer to \cite{PrSi13, PrSi15} for a thorough development of the
material presented below.
Recall that the {\em second order bundle} of $\Gamma$ is
given by
$$\cN^2\Gamma:=\{(p,\nu_\Gamma(p),L_\Gamma(p)):\, p\in\Gamma\}.$$
Note that the Weingarten map $L_\Gamma$ (also called the shape operator,
or the second fundamental tensor) is defined by
$$ L_\Gamma(p) = -\nabla_\Gamma \nu_\Gamma(p),\quad p\in\Gamma ,$$
where $\nabla_\Gamma$ denotes the surface gradient on $\Gamma$.
The eigenvalues $\kappa_j(p)$ of $L_\Gamma(p)$ are the principal curvatures of $\Gamma$ at $p\in\Gamma$, and we have
$|L_\Gamma(p)|=\max_j|\kappa_j(p)|.$
The {\em curvature} $\cH_\Gamma(p)$ is defined by
$$\cH_\Gamma(p) = \sum_{j=1}^{n-1} \kappa_j(p)={\rm tr} L_\Gamma(p) = -{\rm div}_\Gamma \nu_\Gamma(p),$$
where ${\rm div}_\Gamma$ means surface divergence. Recall also that the
{\em Hausdorff distance} $d_H$ between the two closed subsets $A,B\subset\R^m$
is defined by
$$d_H(A,B):= \max\big\{\sup_{a\in A}{\rm dist}(a,B),\sup_{b\in B}{\rm dist}(b,A)\big\}.$$
Then we may approximate $\Gamma$ by a real analytic hypersurface $\Sigma$ (or merely $\Sigma\in C^3$),
in the sense that the Hausdorff distance of the second order bundles of
$\Gamma$ and $\Sigma$ is as small as we want. More precisely, for each $\eta>0$ there is a real analytic closed
hypersurface such that
$d_H(\cN^2\Sigma,\cN^2\Gamma)\leq\eta$. If $\eta>0$ is small enough, then $\Sigma$
bounds a domain $\Omega_1^\Sigma$ with  $\overline{\Omega^\Sigma_1}\subset\Omega$,
and we set $\Omega^\Sigma_2=\Omega\setminus\bar{\Omega}^\Sigma_1$.

It is well known that such a hypersurface $\Sigma$ admits a tubular neighborhood,
which means that there is $a>0$ such that the map
\begin{eqnarray*}
&&\Lambda:\, \Sigma \times (-a,a)\to \R^n \\
&&\Lambda(p,r):= p+r\nu_\Sigma(p)
\end{eqnarray*}
is a diffeomorphism from $\Sigma \times (-a,a)$
onto $\cR(\Lambda)$. The inverse
$$\Lambda^{-1}:\cR(\Lambda)\mapsto \Sigma\times (-a,a)$$ of this map
is conveniently decomposed as
$$\Lambda^{-1}(x)=(\Pi_\Sigma(x),d_\Sigma(x)),\quad x\in\cR(\Lambda).$$
Here $\Pi_\Sigma(x)$ means the nonlinear orthogonal projection of $x$ to $\Sigma$ and $d_\Sigma(x)$ the signed
distance from $x$ to $\Sigma$; so $|d_\Sigma(x)|={\rm dist}(x,\Sigma)$ and $d_\Sigma(x)<0$ iff
$x\in \Omega_1^\Sigma$. In particular we have $\cR(\Lambda)=\{x\in \R^n:\, {\rm dist}(x,\Sigma)<a\}$.

On the one hand, $a$ is determined by the curvatures of $\Sigma$, i.e.\ we must have
$$0<a<\min\big\{1/|\kappa_j(p)|: j=1,\ldots,n-1,\; p\in\Sigma\big\},$$
where $\kappa_j(p)$ mean the principal curvatures of $\Sigma$ at $p\in\Sigma$.
But on the other hand, $a$ is also connected to the topology of $\Sigma$,
which can be expressed as follows. Since $\Sigma$ is a compact (smooth) manifold of
dimension $n-1$ it satisfies a (interior and exterior) ball condition, which means that
there is a radius $r_\Sigma>0$ such that for each point $p\in \Sigma$
there are $x_j\in \Omega_j^\Sigma$, $j=1,2$, such that $B_{r_\Sigma}(x_j)\subset \Omega_j^\Sigma$, and
$\bar{B}_{r_\Sigma}(x_j)\cap\Sigma=\{p\}$. Choosing $r_\Sigma$ maximal,
we then must also have $a<r_\Sigma$.
In the sequel we fix
$$ a= \frac{1}{2}\min\left\{r_\Sigma, \frac{1}{|\kappa_j(p)|}, \, j=1,\ldots, n-1,\; p\in\Sigma\right\}.$$
For later use we note that the derivatives of $\Pi_\Sigma(x)$ and $d_\Sigma(x)$ are given by
$$\nabla d_\Sigma(x)= \nu_\Sigma(\Pi_\Sigma(x)),
\quad \Pi_\Sigma^\prime(x) = M_0(d_\Sigma(x),\Pi(x))P_\Sigma(\Pi_\Sigma(x))$$
for $|d_\Sigma(x)|<a$,
where $P_\Sigma(p)=I-\nu_\Sigma(p)\otimes\nu_\Sigma(p)$ denotes the orthogonal projection onto
the tangent space $T_p\Sigma$ of $\Sigma$ at $p\in\Sigma$, and
\begin{equation}
\label{M-0}
M_0(r)(p)=(I-r L_\Sigma(p))^{-1},\quad (r,p)\in (-a,a)\times\Sigma.
\end{equation}
Note that $$|M_0(r)(p)|\leq 1/(1-r|L_\Sigma(p)|)
\leq 2,\quad \mbox{ for all } (r,p)\in (-a,a)\times\Sigma.$$
Setting $\Gamma=\Gamma(t)$, we may use the map $\Lambda$ to parameterize the unknown free
boundary $\Gamma(t)$ over $\Sigma$ by means of a height function $h(t,p)$ via
$$\Gamma(t)=\{ p+ h(t,p)\nu_\Sigma(p): p\in\Sigma,\; t\geq0\},$$
at least for small $|h|_\infty$.
Extend this diffeomorphism to all of $\bar{\Omega}$ by means of
$$ \Xi_h(t,x) = x +\chi(d_\Sigma(x)/a)h(t,\Pi_\Sigma(x))\nu_\Sigma(\Pi_\Sigma(x))=:x+\xi_h(t,x).$$
Here $\chi$ denotes a suitable cut-off function. More precisely, $\chi\in\cD(\R)$,
$0\leq\chi\leq 1$, $\chi(r)=1$ for $|r|<1/3$, and $\chi(r)=0$ for $|r|>2/3$.
 Note that $\Xi_h(t,x)=x$ for $|d(x)|>2a/3$, and
$$\Xi_h^{-1}(t,x)= x-h(t,x)\nu_\Sigma(x),\quad x\in\Sigma,$$
for $|h|_\infty$ sufficiently small.

Setting
\begin{equation*}
v(t,x)=u(t,\Xi_\rho(t,x))\quad\text{or}\quad u(t,x)= v(t,\Xi_\rho^{-1}(t,x))
\end{equation*}
we have this way transformed the
time varying regions $\Omega\setminus \Gamma(t)$ to the fixed domain
$\Omega\setminus\Sigma$. This is the direct mapping method, also
called Hanzawa transformation.

By means of this transformation, we obtain the following transformed
problem:
\begin{equation}
\label{transformed}
\left\{\begin{aligned}
\kappa(v)\partial_t v +\cA(v,\rho)v&=\kappa(v)\cR(\rho)v
         &&\text{in}&&\Omega\setminus\Sigma\\
\partial_{\nu} v&=0 &&\text{on}&&\partial \Omega\\
 [\![v]\!]=0, \quad v_\Gamma &=v&&\text{on}&&\Sigma \\
[\![\psi(v_\Gamma) ]\!] + \sigma(v_\Gamma) \cH(\rho)- \gamma(v_\Gamma)\beta(\rho)\partial_t\rho&=0
  &&\text{on}&&\Sigma\\
\kappa_\Gamma(v_\Gamma)\partial_t v_\Gamma+\cC(v_\Gamma,\rho)v_\Gamma+\cB(v,\rho)v
&= &&\\
=-\{l(v_\Gamma)+l_\Gamma(v_\Gamma)\cH(\rho)&-\gamma(v_\Gamma)\beta(\rho)\partial_t\rho\}\beta(\rho)\partial_t\rho
 &&\text{on}&&\Sigma\\
v(0)=v_0,\ \rho(0)&=\rho_0.&&
\end{aligned}\right.
\end{equation}
Here $\cA(v,\rho)$, $\cB(v,\rho)$ and $\cC(v_\Gamma,\rho)$ denote the transformed versions of the
operators $-{\rm div}(d\nabla )$,
$-[\![d\partial_\nu]\!]$, and $-{\rm div}_\Gamma(d_\Gamma\nabla_\Gamma)$, respectively. Moreover, $\cH(\rho)$ means the mean curvature of $\Gamma$,
$\beta(\rho)=(\nu_\Sigma|\nu_\Gamma(\rho))$, the term $\beta(\rho)\partial_t\rho$ represents the normal velocity $V$,
and
$$\cR(\rho)v=\partial_t v -\partial_tu\circ \Xi_\rho.$$
The system (\ref{transformed}) is a quasi-linear parabolic problem
on the domain $\Omega$  with fixed interface $\Sigma\subset \Omega$
with {\em dynamic boundary conditions}.

To elaborate on the structure of this problem in more detail, we calculate
$$D\Xi_\rho = I + D\xi_\rho, \quad \quad [D\Xi_\rho]^{-1} = I - {[I + D\xi_\rho]}^{-1}D\xi_\rho
=:I-M_1(\rho)^{\sf T}.$$
where $D$ deontes the derivative with respect to the space variables.
Hence $D\xi_\rho=0$ for $|d_\Sigma(x)|>2a/3$ and
\begin{align*}
D\xi_\rho(t,x)&= \frac{1}{a}\chi'(d_\Sigma(x)/a)\rho(t,\Pi_\Sigma(x))\nu_\Sigma(\Pi_\Sigma(x))\otimes\nu_\Sigma(\Pi_\Sigma(x))\\
& + \chi(d_\Sigma(x)/a)[\nu_\Sigma(\Pi_\Sigma(x))\otimes M_0(d_\Sigma(x))\nabla_\Sigma \rho(t,\Pi_\Sigma(x))]\\
&-\chi(d_\Sigma(x)/a)\rho(t,\Pi_\Sigma(x))L_\Sigma(\Pi_\Sigma(x))M_0(d_\Sigma(x))P_\Sigma(\Pi_\Sigma(x))
\end{align*}
for $0\leq|d_\Sigma(x)|\leq 2a/3$.
In particular, for $x\in \Sigma$ we have
$$D\xi_\rho(t,x)=\nu_\Sigma(x)\otimes\nabla_\Sigma\rho(t,x) -\rho(t,x)L_\Sigma(x)P_\Sigma(x), $$
and
$$[D\xi_\rho]^{\sf T}(t,x)=\nabla_\Sigma\rho(t,x)\otimes\nu_\Sigma(x) -\rho(t,x)L_\Sigma(x), $$
since $L_\Sigma(x)$ is symmetric and has range in $T_x\Sigma$.
Therefore, $[I + D\xi_\rho]$ is boundedly invertible, if $\rho$ and $\nabla_\Sigma \rho$
are sufficiently small, and
\begin{equation*}
\label{hanzawainv}
|{[I + D\xi_\rho]}^{-1}| \leq 2 \quad \mbox{ for }\;
{|\rho|}_\infty \leq \frac{1}{4 ({|\chi'|}_\infty/a+ 2\max_j |\kappa_j|)},
\quad {|\nabla_\Sigma \rho|}_\infty \leq \frac{1}{8}.
\end{equation*}
Employing this notation we obtain
\begin{align*}
\nabla u\circ\Xi_\rho
=  ([D\Xi_\rho^{-1})]^{{\sf T}}\circ\Xi_\rho)\nabla v
=[D\Xi_\rho]^{-1,{\sf T}}\nabla v
= :(I - M_1(\rho))\nabla v,
\end{align*}
and for a vector field $q= \bar{q}\circ \Xi_\rho$
\begin{align*}
(\nabla|\bar{q})\circ\Xi_\rho
=  (([D\Xi_\rho^{-1}]^{{\sf T}}\circ\Xi_\rho)\nabla|q)
=  ([D\Xi_\rho]^{-1,{\sf T}}\nabla|q)
=  ((I - M_1(\rho))\nabla|q).
\end{align*}
Further we have
\begin{align*}
\partial_t u\circ\Xi_\rho
&=  \partial_t v -(\nabla u\circ \Xi_\rho|\partial_t\Xi_\rho)
 =  \partial_t v -( (D\Xi_\rho^{-1}]^{\sf T}\circ\Xi_\rho)\nabla v|\partial_t\Xi_\rho) \\
&=  \partial_t v -([D\Xi_\rho]^{-1,{\sf T}}\nabla  v|\partial_t\xi_\rho )
 =  \partial_t v -(\nabla  v|(I-M_1^{\sf T}(\rho))\partial_t\xi_\rho )
,
\end{align*}
hence
$$\cR(\rho)v=(\nabla  v|(I-M_1^{\sf T}(\rho))\partial_t\xi_\rho ).$$
The normal time derivative transforms as
\begin{equation*}
 \partial_{t,n}u_\Gamma\circ\Xi_\rho =\partial_t v_\Gamma + (\nabla_\Sigma v_\Gamma|\nu_\Sigma)V =\partial_t v_\Gamma,
 \end{equation*}
as $\nabla_\Sigma v_\Gamma$ is perpendicular to $\nu_\Sigma$.
\\
With the Weingarten tensor $L_\Sigma=-\nabla_\Sigma\nu_\Sigma$  we obtain
\begin{equation*}
\begin{aligned}
\nu_\Gamma(\rho)&= \beta(\rho)(\nu_\Sigma-\alpha(\rho)),&& \alpha(\rho)= M_0(\rho)\nabla_\Sigma \rho,\\
M_0(\rho)&=(I-\rho L_\Sigma)^{-1},&& \beta(\rho) = (1+|\alpha(\rho)|^2)^{-1/2},
\end{aligned}
\end{equation*}
and
$$V=(\partial_t\Xi_\rho|\nu_\Gamma) = (\nu_\Sigma|\nu_\Gamma(\rho))\partial_t \rho
=\beta(\rho)\partial_t \rho.$$
For the mean curvature $\cH(\rho)=\cH(\Gamma_\rho)$ we have
$$ \cH(\rho) = \beta(\rho)\{ {\rm tr} [M_0(\rho)(L_\Sigma+\nabla_\Sigma \alpha(\rho))]
-\beta^2(\rho)(M_0(\rho)\alpha(\rho)|[\nabla_\Sigma\alpha(\rho)]\alpha(\rho))\},$$
an expression involving second order derivatives of $\rho$ only linearly. More precisely,
\begin{align*}
\cH(\rho)&=\beta(\rho)\cG(\rho):\nabla_\Sigma^2\rho + \beta(\rho)\cF(\rho),\\
\cG(\rho)&= M_0^2(\rho)-\beta^2(\rho)M_0(\rho)\nabla_\Sigma\rho\otimes M_0(\rho)\nabla_\Sigma\rho.
\end{align*}
Note that $\beta$ as well as $\cF$ and $\cG$ only depend on $\rho$ and $\nabla_\Sigma\rho$.
The linearization of the curvature $\cH(\rho)=\cH(\Gamma_\rho)$ is given by
\begin{equation}
\label{lin-curvature}
\cH^\prime(0)= {\rm tr}\, L_\Sigma^2 +\Delta_\Sigma.
\end{equation}
Here $\Delta_\Sigma$ denotes the Laplace-Beltrami operator on $\Sigma$.
 $\cB(v,\rho)$ becomes
\begin{align*}
\cB(v,\rho)v&= -[\![d(u)\partial_\nu u]\!]\circ\Xi_\rho=-([\![d(v)(I-M_1(\rho))\nabla v]\!]|\nu_\Gamma)\\
&= -\beta(\rho)([\![d(v)(I-M_1(\rho))\nabla v]\!]|\nu_\Sigma -\alpha(\rho))\\
&= -\beta(\rho)[\![d(v)\partial_{\nu_\Sigma} v]\!]
+\beta(\rho)([\![d(v)\nabla v]\!]|(I-M_1(\rho))^{\sf
T}\alpha(\rho)),
\end{align*}
since $M_1^{\sf T}(\rho)\nu_\Sigma = 0$, and
\begin{align*}
\cA(v,\rho)v= & -{\rm div}( d(u)\nabla u)\circ\Xi_\rho= -((I-M_1(\rho))\nabla|d(v)(I-M_1(\rho))\nabla v)\\
= & -d(v)\Delta v + d(v)[M_1(\rho)+M_1^{\sf T}(\rho)-M_1(\rho)M_1^{\sf T}(\rho)]:\nabla^2 v\\
&-d^\prime(v)|(I-M_1(\rho))\nabla v|^2
 + d(v)((I-M_1(\rho)):\nabla M_1(\rho)|\nabla v).
\end{align*}
We recall that for matrices  $A,B\in\R^{n\times n}$, $
A:B=\sum_{i,j=1}^n a_{ij}b_{ij}=\text{tr}\,(AB^{\sf T}) $ denotes
their inner product.
The pull back of $\nabla_\Gamma$ is given by
$$ \nabla_\Gamma \varphi\circ \Xi_\rho = P_\Gamma(\rho) M_0(\rho)\nabla_\Sigma \varphi,$$
where
$$P_\Gamma(\rho) = I -\nu_\Gamma(\rho)\otimes\nu_\Gamma(\rho).$$
This implies  for $\cC(v_\Gamma,\rho)v_\Gamma$ the relation
$$\cC(v_\Gamma,\rho)v_\Gamma=-{\rm tr}\{P_\Gamma(\rho)M_0(\rho)\nabla_\Sigma\big(d_\Gamma(v_\Gamma)P_\Gamma(\rho)M_0(\rho)\nabla_\Sigma v_\Gamma\big)\}.$$
It is easy to see that the leading part of $\cA(v,\rho)v$ is $-d(v)\Delta v$,
while that of $\cB(v,\rho)v$ is
$-\beta(\rho)[\![d(v)\partial_{\nu} v]\!]$,
and the leading part of $\cC(v_\Gamma,\rho)v_\Gamma$ turns out to be 
$-d_\Gamma(v_\Gamma)\Delta_\Sigma v_\Gamma$.
This follows from $M_0(0)=1$, $P_\Gamma(0)=P_\Sigma$, $M_1(0)=0$ and
$\alpha(0)=0$; recall that we may assume $\rho$ small in the
$C^2$-norm. It is important to recognize the quasilinear structure
of \eqref{transformed}.
\section{Linearization at Equilibria}
The full linearization at an equilibrium $(u_*,u_{\Gamma*},\Gamma_*)$ with $u_{\Gamma*}=u_*$, $\Gamma_*=\cup_k \Sigma_k$ a finite union of disjoint spheres contained in $\Omega$ and with radius $R_*>0$ given by
$R_*=(n-1)\sigma(u_*)/[\![\psi(u_*)]\!]$, reads
\begin{equation}
\label{lin}
\left\{\begin{aligned}
\kappa_*\partial_tv-d_*\Delta v&=\kappa_*f &&\text{in} && \Omega\setminus\Gamma_*\\
\partial_{\nu} v&=0 &&\text{on} && \partial \Omega\\
[\![v]\!]=0,v_\Gamma&=v &&\text{on}&& \Gamma_*\\
\kappa_{\Gamma*}\partial_t v_\Gamma -d_{\Gamma*} \Delta_*v_\Gamma-[\![d_*\partial_\nu v]\!]
+l_*u_*\partial_t \rho &=\kappa_{\Gamma*}f_\Gamma &&\text{on}&& \Gamma_*\\
l_* v_\Gamma - \sigma_* A_* \rho -\gamma_* \partial_t\rho&=g &&\text{on}&& \Gamma_*\\
v(0)=v_0,\ \rho(0)&=\rho_0.&& &&
\end{aligned}\right.
\end{equation}
Here
\begin{align*}
&\kappa_*=\kappa(u_*)>0,\quad &\kappa_{\Gamma*}=\kappa_\Gamma(u_*)>0,\quad &d_*=d(u_*)>0,\\
&d_{\Gamma*}=d_\Gamma(u_*)>0,\quad  &\sigma_*=\sigma(u_*)>0,\quad &\gamma_*=\gamma(u_*)\geq0,
\end{align*}
and as in \eqref{l-stern}
\begin{equation*}
 l_*= [\![\psi^\prime(u_*)]\!]+\sigma^\prime(u_*)\cH(\Gamma_*)
= \varphi^\prime(u_*)-\sigma^\prime(u_*)\varphi(u_*)/\sigma(u_*)
=\sigma(u_*)\lambda^\prime(u_*),
\end{equation*}
and
$$ A_*=-(\frac{n-1}{R_*^2}+\Delta_*),$$
where $\Delta_*$
denotes the {\em Laplace-Beltrami} operator on $\Gamma_*$.
\subsection{Maximal Regularity}
We begin with the case $\gamma_*>0$, which is the simpler one. Define the operator $L$ in
$$X_0:= L_p(\Omega)\times W^{r}_p(\Gamma_*)\times W^{s}_p(\Gamma_*)$$
with
$$X_1:=W^2_p(\Omega\setminus\Gamma_*)\times W^{2+r}_p(\Gamma_*)
\times W^{2+s}_p(\Gamma_*)$$
by means of
\begin{equation*}
\label{L-gamma}
\begin{split}
D(L)&=
\big\{(v,v_\Gamma,\rho)\in X_1:\; [\![v]\!]=0,\;\; v_\Gamma=v\text{ on }\Gamma_*,
\;\; \partial_\nu v=0 \text{ on } \partial\Omega\big\}, \\
L&=\left[\begin{array}{ccc}
 \!\! (-d_*/\kappa_*)\Delta &0 &0\\
\!\! -[\![(d_*/\kappa_{\Gamma*})\partial_\nu ]\!] &(l_*^2u_*/\gamma_*-d_{\Gamma*}\Delta_*)/\kappa_{\Gamma*} &- (l_*u_*\sigma_*/\gamma_*\kappa_{\Gamma*})A_*\\
\!\!  0&-(l_*/\gamma_*)&(\sigma_*/\gamma_*)A_*\end{array}\right]
\end{split}
\end{equation*}
In case $\gamma_*>0$, problem \eqref{lin} is equivalent to
the Cauchy problem
\begin{equation*}
\dot z+L z=(f,f_\Gamma-(l_*u_*/\gamma_*\kappa_{\Gamma*})g,g),
\quad z(0)=z_0,
\end{equation*}
where $z=(v,v_\Gamma,\rho)$ and $z_0=(v_0,v_0|_{\Gamma_0},\rho_0)$.
The main result on problem (\ref{lin}) for $\gamma_*>0$ is the following.
\begin{theorem}
\label{MR-gamma}
Let $1<p<\infty$,  $\gamma_*>0$, and
$$ -1/p\leq r\leq 1-1/p,\quad r\leq s\leq r+2.$$
 Then for each finite interval $J=[0,a]$, there is a unique solution
\begin{equation*}
 (v,v_\Gamma,\rho)\in\EE(J):=H^1_p(J;X_0)\cap L_p(J;X_1)
\end{equation*}
of \eqref{lin} if and only if the data $(f,f_\Gamma,g)$ and $(v_0,v_{\Gamma0},\rho_0)$ satisfy
\begin{equation*}
\begin{split}
&(f,f_\Gamma,g)\in\FF(J)= L_p(J;X_0)^3,\\
&(v_0,v_{\Gamma0},\rho_0)\in W^{2-2/p}_p(\Omega\setminus\Gamma_*)\times W^{2+r-2/p}_p(\Gamma_*)\times W^{2+s-2/p}_p(\Gamma_*)
\end{split}
\end{equation*}
and the compatibility conditions
\begin{equation*}
[\![v_0]\!]=0,\quad v_{\Gamma0}=v_0\;\;{\rm on}\;\;\Gamma_*,
\quad\partial_\nu v=0\;\;{\rm on}\;\;\partial\Omega.
\end{equation*}
The operator $-L$ defined above generates an analytic $C_0$-semigroup in
$X_0$ with maximal regularity of type $L_p$.
\end{theorem}
\begin{proof}
Looking at the entries of $L$ we see that $L:X_1\to X_0$ is bounded provided $r\leq 1-1/p$, $r\leq s$, 
and $ s\leq r+2$. The compatibility condition $v_\Gamma=v_{|_{\Gamma_*}}$ implies $r+2\geq 2-1/p$. This explains the constraints on the parameters $r$ and $s$. To obtain maximal $L_p$-regularity, we first consider the case $s>r$. Then $L$ is lower triangular up to a perturbation. So we may solve the problem for $(v,v_\Gamma)$ with maximal $L_p$-regularity (cf.\ \cite{DPZ08} for the one-phase case)
 first and then that for $\rho$. In the other case we have $r=s$. Then the second term in the third line in the definition of $L$ is of lower order, hence $\rho$ decouples from $(v,v_\Gamma)$. This way we also obtain maximal $L_p$-regularity.
Since the Cauchy problem for $L$ has  maximal $L_p$-regularity, we can now infer from
\cite[Proposition 1.2]{Pru03} that $-L$ generates an analytic
$C_0$-semigroup in $X_0$.
\end{proof}

We note that if $l_*=0$ and $\gamma_*=0$  then the linear problem \eqref{lin} is not
well-posed. In fact, in this case the linear Gibbs-Thomson relation reads
$$ -\sigma_* A_* \rho =g,$$
which is not well-posed as the kernel of $A_*$ is non-trivial and $A_*$ is not surjective.

Now we consider the case $l_*\neq0$ and $\gamma_*=0$.
For the solution space we fix again $r,s\in\R$ with $r\leq s\leq r+2$, $-1/p\leq r\leq 1-1/p$, and consider
$$(v,v_\Gamma,\rho)\in \EE(J)=H^1_p(J,X_0)\cap L_p(J;X_1).$$
Then by trace theory the space of data becomes
\begin{align*}
(f,f_\Gamma,g)\in\FF_0(J)&:= L_p(J;L_p(\Omega))\times L_p(J;W^{r}_p(\Gamma_*))\\
 &\times[H^1_p(J;W^{s-2}_p(\Gamma_*)\cap L_p(J;W^{s}_p(\Gamma_*))],
\end{align*}
and the space of initial values will be
$$ (v_0,v_{\Gamma0},\rho_0)\in W^{2-2/p}_p(\Omega\setminus\Gamma_*)\times W^{r+2-2/p}_p(\Gamma_*)
\times W^{s+2-2/p}_p(\Gamma_*)$$
with compatibilities
$$ [\![v_0]\!]=0,\quad v_{\Gamma0}=v_0,\quad 
l_*v_{\Gamma0} - \sigma_*A_*\rho_0=g(0)\;\;\text{on}\;\;\Gamma_*, 
\quad \partial_\nu v=0 \;\;\text{on}\;\; \partial\Omega.$$
To obtain maximal $L_p$-regularity, we replace $v_\Gamma$ by  the  Gibbs-Thomson relation,
which for $\gamma_*=0$ is an elliptic equation. We obtain
$v_\Gamma = (\sigma_*/l_*) A_* \rho +g/l_*.$
Inserting this expression into the energy balance on the surface $\Gamma_*$ yields
\begin{equation}
\label{eq}
\big(l_*u_* + (\kappa_{\Gamma*}\sigma_*/l_*)A_*\big)\partial_t\rho
-d_{\Gamma*}\Delta_*v_\Gamma-[\![d_*\partial_\nu v]\!]
 =\kappa_{\Gamma*}( f_\Gamma -\partial_t g/l_*).
\end{equation}
Moreover, we obtain
\begin{equation*}
\begin{split}
d_{\Gamma*}\Delta_*v_\Gamma
&=(l_*u_*+(\kappa_{\Gamma*}\sigma_*/l_*)A_*))(d_{\Gamma*}/\kappa_{\Gamma*})\Delta_*\rho\\
&-(l_*u_*d_{\Gamma*}/\kappa_{\Gamma*})\Delta_*\rho+(d_{\Gamma*}/l_*)\Delta_*g.
\end{split}
\end{equation*}
Now we assume that
\begin{equation}
\label{exceptional}
 \eta_*:=\frac{(n-1) \sigma_*\kappa_{\Gamma*}}{u_*l_*^2 R_*^2}\neq 1
\end{equation}
which is equivalent to invertibility of the operator  $A_0:=l_*u_* + (\kappa_{\Gamma*}\sigma_*/l_*)A_*$. Applying its inverse to \eqref{eq} we arrive at the following equation for $\rho$:
\begin{equation}
\label{rho-eq}
\begin{aligned}
\partial_t\rho-(d_{\Gamma*}/\kappa_{\Gamma*})\Delta_*\rho+A_0^{-1}\{(u_*l_*d_{\Gamma*}/\kappa_{\Gamma*})\Delta_*\rho-[\![d_*\partial_\nu v]\!]\}=\tilde{g},
\end{aligned}
\end{equation}
with
$$\tilde{g}=A_0^{-1}\big\{\kappa_{\Gamma*} f_\Gamma
-((\kappa_{\Gamma*}/l_*)\partial_t g -(d_{\Gamma*}/l_*)\Delta_*g)\big\}.$$
Solving equation \eqref{eq} for $\partial_t\rho$ we obtain
for $v_\Gamma$:
\begin{align}
\label{v-Gamma-eq}
\kappa_{\Gamma*}\partial_t v_\Gamma -d_{\Gamma*} \Delta_*v_\Gamma-[\![d_*\partial_\nu v]\!] + l_*u_*A_0^{-1}\{d_{\Gamma*} \Delta_*v_\Gamma+[\![d_*\partial_\nu v]\!]\}= \tilde{f}_\Gamma.
\end{align}
where
\begin{equation*}
\tilde{f}_\Gamma = \kappa_{\Gamma*}\{f_\Gamma-l_*u_*A_0^{-1}(f_\Gamma-\partial_tg/l_*)\}.
\end{equation*}
Then by the regularity of $f_\Gamma$ and $g$ and with $r\leq s\leq r+2$ we see that
\begin{equation*}
\quad \tilde{f}_\Gamma\in L_p(J;W^r_p(\Gamma_*)),
\quad \tilde{g}\in  L_p(J;W^{s}_p(\Gamma_*)).
\end{equation*}
So the linear problem \eqref{lin} can be recast as an evolution equation in $X_0$ as
$$\dot{z}+L_0 z = (f,\tilde f_\Gamma,\tilde g),\quad z(0)=z_0,$$
with $L_0=L_{00}+L_{01}$ defined by
$$D(L_{0j})=\big\{(v,v_\Gamma,\rho)\in X_1:\; [\![v]\!]=0,\;\:
v_\Gamma=v \text{ on } \Gamma_*,\;\; \partial_\nu v=0 \text{ on }\partial\Omega\big\},$$
and
\begin{equation*}
\label{L-00}
\begin{split}
L_{00}&=\left[\begin{array}{ccc} (-d_*/\kappa_*)\Delta &0&0\\
 -[\![(d_*/\kappa_{\Gamma*})\partial_\nu ]\!]&-(d_{\Gamma*}/\kappa_{\Gamma*})\Delta_*&0 \\
   -A_0^{-1}[\![d_*\partial_\nu]\!]&0&-(d_{\Gamma*}/\kappa_{\Gamma*}) \Delta_*\end{array}\right],
\end{split}
\end{equation*}
and
\begin{equation*}
\label{L-01}
\begin{split}
L_{01}&=\!
\left[\begin{array}{ccc} 0 &0&0\\
 \!\!(l_*u_*/\kappa_{\Gamma*})A_0^{-1}[\![d_*\partial_\nu ]\!]&(l_*u_*d_{\Gamma*}/\kappa_{\Gamma*})A_0^{-1} \Delta_*&0 \\
 0&0& (u_*l_*d_{\Gamma*}/\kappa_{\Gamma*})A_0^{-1}\Delta_*
 \end{array}\right]\!.
\end{split}
\end{equation*}
Looking at $L_0$ we first note that $L_{01}$ is a lower order perturbation of $L_{00}$. The latter is lower triangular, and the problem for $(v,v_\Gamma)$ as above has maximal $L_p$-regularity in $X_0$. As the diagonal entry in the equation for $\rho$ has maximal $L_p$-regularity as well we may conclude that $-L_0$
generates an analytic $C_0$-semigroup with maximal regularity in $X_{0}$
More precisely, we have the following result.
\begin{theorem}
\label{MR-0}
 Let $1<p<\infty$,  $\gamma_*=0$, $-1/p\leq r\leq1-1/p$, $r\leq s\leq r+2$,
$l_*\neq0$, and assume $u_*l_*^2R_*^2\neq\kappa_{\Gamma*}\sigma_*(n-1)$.

Then for each interval $J=[0,a]$, there is a unique solution $(v,v_\Gamma,\rho)\in \EE(J)$
of (\ref{lin}) if and only if the data $(f,f_\Gamma,g)$ and $(v_0,v_{\Gamma0},\rho_0)$ satisfy
\begin{equation*}
\begin{split}
&(f,f_\Gamma,g)\in\FF_0(J),\\
&(v_0,v_{\Gamma0},\rho_0)\in W^{2-2/p}_p(\Omega\setminus\Gamma_*)
\times W^{r+2-2/p}_p(\Gamma_*)\times W^{s+2-2/p}_p(\Gamma_*)
\end{split}
\end{equation*}
and the compatibility conditions
$$[\![v_0]\!]=0,\quad v_{\Gamma0}=v_0,\quad l_*v_0-\sigma_* A_* \rho_0=g(0)
\text{ {\rm on} }\Gamma_*,\quad \partial_\nu v=0 \text{ {\rm on} }\partial\Omega.$$
The operator $-L_0$ defined above generates an analytic $C_0$-semigroup in $X_{0}$ with
maximal regularity of type $L_p$.
\end{theorem}
Note that the compatibility condition $l_*v_0-\sigma_* A_* \rho_0=g(0)$ allows to recover the
Gibbs-Thomson relation from the dynamic equations.
Indeed, it follows from \eqref{lin} and \eqref{rho-eq}-\eqref{v-Gamma-eq} that
the function $w:=v_\Gamma-((\sigma_*/l_*)A_*\rho +g/l_*)$ satisfies
the parabolic equation
\begin{equation}
\label{w-equation}
\kappa_{\Gamma*}\partial_tw-d_{\Gamma*}\Delta_*w =0,
\quad w(0)=0 \quad\text{on}\quad \Gamma_*.
\end{equation}
As $w\equiv0$ is the unique solution of \eqref{w-equation}
we conclude that the Gibbs-Thomson relation is satisfied.

\subsection{The Eigenvalue Problem}
By compact embedding, the spectrum of $L$ consists only of countably many discrete eigenvalues of
finite multiplicity and is independent of $p$. Therefore it is enough to consider the case $p=2$.
In the following, we will use the notation
\begin{equation*}
\begin{split}
&(u|v)_\Omega:=(u|v)_{L_2(\Omega)}:=\int_\Omega u\bar v\,dx,
\quad u,v\in L_2(\Omega),\\
&(g|h)_{\Gamma_*}\!:=(g|h)_{L_2(\Gamma_*)}\!:=\int_{\Gamma_*} g\bar h\,ds,
\quad g,h\in L_2(\Gamma_*),
\end{split}
\end{equation*}
for the $L_2$ inner product in $\Omega$ and $\Gamma_*$, respectively.
Moreover, we set $|v|_\Omega=(v|v)^{1/2}_{\Omega}$ and
$|g|_{\Gamma_*}=(g|g)^{1/2}_{\Gamma_*}$.
The eigenvalue problem reads as follows:
\begin{equation}
\label{evp}
\left\{\begin{aligned}
\kappa_*\lambda v-d_*\Delta v&=0 &&\text{in }&&\Omega\setminus\Gamma_*\\
\partial_{\nu} v &=0 &&\text{on }&&\partial \Omega\\
[\![v]\!]&=0 &&\text{on }&&\Gamma_*\\
l_*v-\sigma_* A_* \rho -\gamma_* \lambda\rho &=0 &&\text{on }&&\Gamma_*\\
\kappa_{\Gamma*}\lambda v -d_{\Gamma*}\Delta_*v-[\![d_*\partial_\nu v]\!]+l_*u_*\lambda \rho&=0 &&\text{on }&&\Gamma_*.\\
\end{aligned}\right.
\end{equation}
Let $\lambda\neq0$ be an eigenvalue with eigenfunction $(v,\rho)\neq0$.
Then (\ref{evp}) yields
\begin{align*}
0=\lambda|\sqrt{\kappa_*}v|^2_{\Omega}-(d_*\Delta v|v)_\Omega
= \lambda|\sqrt{\kappa_*}v|^2_{\Omega}+ |\sqrt{d_*}\nabla v|_\Omega^2 +([\![d_*\partial_\nu v]\!]|v)_{\Gamma_*}.
\end{align*}
On the other hand, we have on the interface
\begin{align*}
0&=\kappa_{\Gamma*}\lambda|v|^2_{\Gamma_*}-d_{\Gamma*}(\Delta_\Gamma v|v)_{\Gamma_*}-([\![d_*\partial_\nu v]\!]|v)_{\Gamma_*}+\lambda u_*l_*(\rho|v)_{\Gamma_*}\\
&=  \lambda\kappa_{\Gamma*}|v|^2_{\Gamma_*}+d_{\Gamma*}|\nabla_\Gamma v|_{\Gamma_*}^2-([\![d_*\partial_\nu v]\!]|v)_{\Gamma_*}
+\lambda u_*l_*(\rho|v)_{\Gamma_*}.
\end{align*}
Adding these identities we obtain
\begin{align*}
0= \lambda|\sqrt{\kappa_*}v|^2_{\Omega}+ |\sqrt{d_*}\nabla v|_\Omega^2 +\lambda\kappa_{\Gamma*}|v|^2_{\Gamma_*}+d_{\Gamma*}|\nabla_\Gamma v|_{\Gamma_*}^2
+\lambda u_*l_*(\rho|v)_{\Gamma_*},
\end{align*}
hence employing the Gibbs-Thomson law this results into the relation
\begin{align*}
 \lambda|\sqrt{\kappa_*}v|^2_{\Omega}+ |\sqrt{d_*}\nabla v|_\Omega^2 +\lambda\kappa_{\Gamma*}|v|^2_{\Gamma_*}&+d_{\Gamma*}|\nabla_\Gamma v|_{\Gamma_*}^2\\
&+\lambda u_*\sigma_* (A_*\rho|\rho)_{\Gamma_*} +
\gamma_*u_*|\lambda|^2|\rho|^2_{\Gamma_*}=0.
\end{align*}
 Since $A_*$ is
selfadjoint in $L_2(\Gamma_*)$, this identity shows that all
eigenvalues of $L$ are real.
Decomposing $v=v_0+\bar{v}$, $v_\Gamma=v_{\Gamma,0}+\bar{v}_\Gamma$, $\rho=\rho_0+\bar{\rho}$, with the normalizations
$(\kappa_*|v_0)_{\Omega}=(v_{\Gamma,0}|1)_{\Gamma_*}=(\rho_0|1)_{\Gamma_*}=0$, this
identity can be rewritten as
\begin{equation}
\label{evid}
\begin{aligned}
&\lambda\big\{|\sqrt{\kappa_*}v_0|^2_{\Omega} + \kappa_{\Gamma*}|v_{\Gamma,0}|^2_{\Gamma_*}+\sigma_*
u_*(A_*\rho_0|\rho_0)_{\Gamma_*}+\lambda u_*\gamma_*|\rho_0|^2_{\Gamma_*}\big\}\\
&+|\sqrt{d_*}\nabla v_0|^2_{\Omega}+ d_{\Gamma*}|\nabla_\Gamma v_{\Gamma,0}|^2_{\Gamma_*}\\
&+\lambda\Big[ (\kappa_*|1)\bar{v}^2 +\kappa_{\Gamma*}|\Gamma_*|\bar{v}_\Gamma^2 -\sigma_*u_*\frac{n-1}{R_*^2}|\Gamma_*|\bar{\rho}^2 + \lambda u_* \gamma_*|\Gamma_*| \bar{\rho}^2\Big] =0.\nonumber
\end{aligned}
\end{equation}
In case  $\Gamma_*$ is connected, $A_*$ is positive semi-definite on functions with mean zero, and hence
the bracket determines whether there are positive eigenvalues.
Taking the mean in \eqref{evp} we obtain
$$(\kappa_*|1)_\Omega \bar{v} + \kappa_{\Gamma*}|\Gamma_*| \bar{v}_\Gamma + l_*u_*|\Gamma_*| \bar{\rho}=0.$$
Hence minimizing the function
$$\phi(\bar{v},\bar{v}_\Gamma,\bar{\rho}):=(\kappa_*|1)\bar{v}^2 +\kappa_{\Gamma*}|\Gamma_*|\bar{v}_\Gamma^2 -\sigma_*u_*\frac{n-1}{R_*^2}|\Gamma_*|\bar{\rho}^2$$
with respect to the constraint we see that there are no positive eigenvalues provided the stability condition  $\zeta_*\leq 1$ is satisfied.
\smallskip

If $\Gamma_*=\bigcup\nolimits_{1\le l\le m}\Gamma^l_\ast$ consists of $m$ spheres $\Gamma^l_\ast$ of equal radius, then
\begin{equation}
\label{N-L-m}
N(L) = {\rm span}\left\{(\frac{\sigma_* (n-1)}{R^2_*},-l_*),(0,Y^l_1),\ldots,(0,Y^l_n) :
\, 1\le l\le m\right\},
\end{equation}
where the functions $Y^l_j$ denote the {\em spherical harmonics of degree one} on $\Gamma_\ast^l$
(and $Y^l_j\equiv 0$ on $\bigcup\nolimits_{i\neq l}\Gamma^i_\ast$),
normalized by $(Y^l_j|Y^l_k)_{\Gamma^l_*}=\delta_{jk}$.
$N(L)$ is isomorphic to the tangent space of $\cE$ at $(u_*,\Gamma_*)\in\cE$,
as was shown in \cite[Theorem 4.5.(vii)]{PSZ10}.

We can now state the main result on linear stability.

\goodbreak
\begin{theorem}
\label{lin-stability}
Let $\sigma_*>0$, $\gamma_*\geq0$, $l_*\neq0$,
\begin{equation*}
\eta_*:=(n-1)\sigma_*\kappa_{\Gamma*}/u_*l_*^2R_*^2\neq 1
\quad\text{in case $\gamma_*=0$},
\end{equation*}
and assume that the interface $\Gamma_*$ consists of $m\geq1$
components. Let
\begin{equation*}
\label{zeta}
 \zeta_\ast= \frac{(n-1)\sigma_* [(\kappa_*|1)_{\Omega}+\kappa_{\Gamma*}|\Gamma_*|]}{u_*l_*^2R_*^2|\Gamma_*|},
\end{equation*}
and let the equilibrium energy ${\sf E}_e$ be defined as in \eqref{eq-energy}. Then
\begin{itemize}
\item[{(i)}]
${\sf E}_e^\prime(u_*)= (\zeta_\ast-1) u_*l^2_*R^2_*|\Gamma_*|/(n-1){\sigma_*}$.
\vspace{0mm}
\item[{(ii)}] $0$ is a an eigenvalue of $L$
with geometric multiplicity $(mn+1)$.
\vspace{0mm}
\item[{(iii)}] $0$ is semi-simple if $\zeta_\ast\neq 1$.
\vspace{0mm}
\item[{(iv)}] If $\Gamma_*$ is connected and $\zeta_\ast\le 1$, or if $\eta_*>1$ and $\gamma_*=0$, then all eigenvalues of $-L$ are negative, except for the eigenvalue $0$.
\vspace{0mm}
\item[{(v)}] If $\zeta_\ast>1$, and  $\eta_*<1$ in case $\gamma_*=0$, then there are precisely $m$ positive eigenvalues of
$-L$, where $m$ denotes the number of equilibrium spheres.
\vspace{0mm}
\item[(vi)] If $\zeta_\ast\leq 1$, and $\eta_*<1$ in case $\gamma_*=0$ then $-L$ has precisely $m-1$ positive eigenvalues.
\item[{(vii)}] $N(L)$ is isomorphic to the tangent space $T_{(u_*,\Gamma_*)}\cE$ of $\cE$ at
$(u_*,\Gamma_*)\in\cE$.
\vspace{0mm}
\end{itemize}
\end{theorem}
\begin{remarks} (a) Formally, the result is also true if $l_*=0$ and
$\gamma_*>0 $. In this case ${\sf E}_e'(u_*)=
(\kappa_*|1)_{\Omega}+\kappa_{\Gamma*}|\Gamma_*|>0$ and $\zeta_\ast=\infty$, hence the
equilibrium is unstable.
If in addition $\gamma_*=0$, then the problem is not well-posed.
\vspace{1mm}\\
(b) Note that $\zeta_\ast$ does
neither depend on the diffusivities $d_*$, $d_{\Gamma*}$, nor on the coefficient of
undercooling  $\gamma_*$.
\vspace{1mm}\\
(c)
It is shown in \cite{PrSi08} that in case
$\zeta_\ast=1$ and $\Gamma_*$  connected,  the eigenvalue $0$ is no longer semi-simple: its
algebraic multiplicity rises by $1$ to $(n+2)$.
\vspace{1mm}\\
\goodbreak
\noindent
(d) It is remarkable that in case kinetic undercooling is absent, large surface heat capacity,
i.e.\ $\eta_*>1$, stabilizes the system, even in such a way that multiple spheres are stable, in contrast to the case $\eta_*<1$.
\vspace{1mm}\\
(e) We can show that, in case $\gamma_*=0$, if $\eta_*$ increases to $1$ then all positive eigenvalues go to $\infty$.
\end{remarks}
We recall a result on the Dirichlet-to-Neumann operator $D_\lambda$, $\lambda\geq0$ which is defined as follows.
Let $g\in H^{3/2}_2(\Gamma_*)$ be given. Solve the elliptic transmission problem
\begin{equation}
\label{ell-tram}
\left\{
\begin{aligned}
\kappa_*\lambda w -d_*\Delta w &=0\quad \text{in}\; \Omega\setminus\Gamma_*, \\
\partial_\nu w&=0\quad \text{on}\; \partial\Omega,\\
[\![w]\!]&=0\quad \text{on}\; \Gamma_*, \\
 w&=g\quad \text{on}\; \Gamma_*,
\end{aligned}
\right.
\end{equation}
and define $D_\lambda g=-[\![d\partial_\nu w]\!]\in H^{1/2}_2(\Gamma_*)$.
\goodbreak
\begin{lemma}\label{DN-op}
The Dirichlet-to-Neumann operator $D_\lambda$ has the following well-known properties.
\begin{itemize}
\item[(a)] $(D_\lambda g|g)_{\Gamma_*} = \lambda |\kappa_*^{1/2}w|^2_\Omega + |d_*^{1/2}\nabla w|_\Omega^2$, for all $g\in H^{3/2}_2(\Gamma_*)$;
\vspace{1mm}
\item[(b)] $|D_\lambda g|_{\Gamma_*}\leq C[\lambda^{1/2}|g|_{\Gamma_*}+|g|_{H^1_2(\Gamma_*)}]$,
for all $g\in H^{3/2}_2(\Gamma_*)$ and $\lambda\ge 1$;
\vspace{1mm}
\item[(c)] $(D_\lambda g|g)_{\Gamma_*}\geq c\lambda^{1/2}|g|_{\Gamma_*}^2$, for all $g\in H^{3/2}_2(\Gamma_*)$ and $\lambda\ge 1$.
\end{itemize}
In particular, $D_\lambda$ extends to a self adjoint positive definite linear operator in $L_2(\Gamma_*)$ with domain $H^1_2(\Gamma_*)$.
\end{lemma}
\subsection{Proof of Theorem \ref{lin-stability}}
For the case that $\kappa_{\Gamma*}=d_{\Gamma*}=0$  this result is proved in \cite{PSZ10}.
Assertion (i) follows from the considerations in part (g) of the introduction.
Assertions (ii), (iii), and (vii) only involve the kernel of $L$ and the manifold of equilibria. 
Since both are the same as in the case $\kappa_{\Gamma*}=d_{\Gamma*}=0$, 
the proofs of (ii), (iii) and (vii) given in \cite{PSZ10} remain valid in the more general situation considered here. 
The first part of assertion (iv) has been proved above,
and it thus remains to prove the assertions in (v) and (vi), and the second part of (iv).

If the stability condition $\zeta_\ast\leq1$ does not hold or if ${\Gamma_*}$ is disconnected, then there
is always a positive eigenvalue. It is a delicate task to prove this. The principal idea to attack this problem is as follows: suppose $\lambda>0$ is an eigenvalue, and
that $\rho$ is known; solve the resolvent diffusion problem
\begin{equation}
\label{NDdiffusion}
\left\{\begin{aligned}
\kappa_*\lambda v -d_*\Delta v &=0 &&\text{in}&& \Omega\setminus{\Gamma_*}\\
\partial_{\nu} v&=0 &&\text{on}&&\partial\Omega \\
[\![v]\!]&=0 &&\text{on} && {\Gamma_*}\\
 v&= v_\Gamma &&\text{on}&&  {\Gamma_*}\\
\end{aligned}\right.
\end{equation}
to get $-[\![d_*\partial_\nu v]\!]=:D_\lambda v_\Gamma$.
Next we solve the resolvent surface diffusion problem
$$ \lambda\kappa_{\Gamma*} v_\Gamma -d_{\Gamma*}\Delta_* v_\Gamma + D_\lambda v_\Gamma= h,$$
to the result
$$v_\Gamma = T_\lambda h:=(\lambda\kappa_{\Gamma*} -d_{\Gamma*}\Delta_* + D_\lambda)^{-1}h.$$
Setting $h=-\lambda u_*l_*\rho$ this implies with
the linearized Gibbs-Thomson law
 the equation
\begin{equation}\label{Tlambda}
[(l_*^2u_*)\lambda T_\lambda +\gamma_* \lambda]\rho +\sigma_* A_* \rho =0.
\end{equation}
$\lambda>0$ is an eigenvalue of $-L$ if and only if
\eqref{Tlambda} admits a nontrivial solution. We consider this
problem in $L_2({\Gamma_*})$. Then $A_*$ is selfadjoint in $L_2({\Gamma_*})$ and
$$\sigma_*(A_* \rho|\rho)_{\Gamma_*} \geq -\frac{(n-1)\sigma_*}{R_*^2} |\rho|^2_{\Gamma_*},$$
for each $\rho\in D(A_*)=H^2_2(\Gamma_*)$.
Moreover, since $A_*$ has compact resolvent, the operator
\begin{equation}
\label{B-lambda}
B_\lambda:=[(l_*^2u_*)\lambda T_\lambda +\gamma_* \lambda]+\sigma_* A_*
\end{equation}
has compact resolvent as well, for each $\lambda>0$. Therefore the spectrum of $B_\lambda$ consists only of eigenvalues which, in addition, are real.
We intend to prove that in case either $\Gamma_*$ is disconnected or the stability condition does not hold, $B_{\lambda_0}$ has $0$ as an eigenvalue, for some $\lambda_0>0$. This has been achieved in \cite{PSZ10} in the simpler case where $\kappa_{\Gamma*}=d_{\Gamma*}=0$, in which case $T_\lambda$ is the Neumann-to-Dirichlet operator for \eqref{NDdiffusion}. Here we try to use similar ideas as in \cite{PSZ10}, namely we investigate $B_\lambda$ for small and for large values of $\lambda$. However, in the situation of this paper this will be more involved.

For this purpose we need more information about $T_\lambda$. So we first consider the problem
\begin{equation}
\label{NDBdiffusion}
\left\{\begin{aligned}
\kappa_*\lambda v -d_*\Delta v &=0 &&\text{in}&& \Omega\setminus{\Gamma_*}\\
\partial_\nu v&=0 &&\text{on}&&\partial\Omega \\
[\![v]\!]&=0 &&\text{on} && {\Gamma_*}\\
\lambda\kappa_{\Gamma*} v -d_{\Gamma*}\Delta_* v -[\![d_*\partial_\nu v]\!] &= g &&\text{on}&&  {\Gamma_*}.\\
\end{aligned}\right.
\end{equation}
As we have seen above this problem has a unique solution for each $\lambda>0$,
denoted by $v=S_\lambda g$. Obviously for $\lambda=0$ this problem has a one-dimensional eigenspace spanned by the constant function ${\sf e}\equiv 1$. The problem is solvable if and only if the mean value of $g$ is zero, i.e.\ if $g\in L_{2,0}(\Gamma_*)$. This implies by compactness that $S_\lambda g\to S_0g$ as well as $T_\lambda\to T_0g$ as $\lambda\to 0^+$, whenever $g$ has mean zero, where $S_0g$ means the unique solution of \eqref{NDBdiffusion} for $\lambda=0$ with mean zero.

\medskip

\noindent (a) \, Suppose that $\Gamma_*$ is disconnected. If the interface
$\Gamma_*$ consists of $m$ components $\Gamma_*^k$, $k=1,..., m$,
we set ${\sf e}_k=1$ on $\Gamma_*^k$ and zero elsewhere. Let
$\rho=\sum_k a_k {\sf e}_k\neq0$ with $\sum_k a_k=0$, hence $Q_0\rho=\rho$,
where $Q_0$ is the canonical projection onto $L_{2,0}(\Gamma_*)$ in $L_2(\Gamma_*)$,
$
Q_0 \rho:=\rho-(\rho|{\sf e})_{\Gamma_*}/|\Gamma_*|.$
Then
$$\lim_{\lambda\to0} \lambda T_\lambda \rho = \lim_{\lambda\to0} \lambda T_\lambda Q_0\rho=0,$$
since $T_\lambda Q_0$ is bounded as $\lambda\to0$. This implies
$$\lim_{\lambda\to0} (B_\lambda \rho|\rho)_{\Gamma_*} = - \,({(n-1)\sigma_*}/{R_*^2})\, \sum_k |\Gamma_*^k|a_k^2<0.$$
Therefore $B_\lambda$ is not positive semi-definite for small $\lambda$.

\medskip

\noindent
(b) \, Suppose next that $\Gamma_*$ is connected. Consider $\rho={\sf e}$. Then we have
$$(B_\lambda {\sf e}|{\sf e})_{\Gamma_*}= u_*l_*^2\lambda(T_\lambda{\sf e}|{\sf
e})_{\Gamma_*} +\lambda\gamma_*|{\sf e}|^2_{\Gamma_*}-
({(n-1)\sigma_*}/{R_*^2})|{\sf e}|_{\Gamma_*}^2.$$ We compute the
limit $\lim_{\lambda \to0}\lambda(T_\lambda{\sf e}|{\sf
e})_{\Gamma_*}$ as follows. First solve the problem
\begin{equation}
\label{NDBdiffusion0}
\left\{\begin{aligned}
 -d_*\Delta v &=-\kappa_* a_0 &&\text{in} && \Omega\setminus{\Gamma_*} \\
\partial_{\nu} v&=0 &&\text{on}&& \partial\Omega \\
[\![v]\!]&=0 &&\text{on}&& {\Gamma_*} \\
 -d_{\Gamma*} \Delta_\Gamma v-[\![d_*\partial_\nu v]\!]&=
 {\sf e}-\kappa_{\Gamma*}a_0 &&\text{on}&&{\Gamma_*},\\
\end{aligned}\right.
\end{equation}
where $a_0=|\Gamma_*|/[(\kappa_*|1)_{\Omega}+\kappa_{\Gamma*}|\Gamma_*|]$, which is solvable
since the necessary compatibility condition holds. Let $v_0$ denote
the solution which satisfies the normalization condition
$(\kappa_*|v_0)_{\Omega}+\kappa_{\Gamma*}(v_0|1)_{\Gamma_*}=0$.
Then $v_\lambda:=S_\lambda{\sf e}-v_0- a_0/\lambda$ satisfies the problem
\begin{equation}
\label{NDBdiffusion1}
\left\{\begin{aligned}
\kappa_* \lambda v_\lambda -d_*\Delta v_\lambda &=-\kappa_*\lambda v_0
&&\text{in}&&\Omega\setminus{\Gamma_*}\\
\partial_{\nu} v&=0 &&\text{on}&&\partial\Omega\\
[\![v_\lambda]\!]&=0 &&\text{on}&& {\Gamma_*}\\
 \kappa_{\Gamma*}\lambda v_\lambda -d_{\Gamma*} \Delta_* v_\lambda-[\![d_*\partial_\nu v_\lambda]\!]&
 = -\lambda\kappa_{\Gamma*}v_0 &&\text{on}&&{\Gamma_*}.\\
\end{aligned}\right.
\end{equation}
By the normalization $(\kappa_*|v_0)_{\Omega}+\kappa_{\Gamma*}(v_0|1)_{\Gamma_*} =0$  we see that the compatibility condition for \eqref{NDBdiffusion} holds for each $\lambda>0$, and so we conclude that
$v_\lambda$ is bounded in $W^2_2(\Omega\setminus\Gamma_*)$ as
$\lambda\to0$, it even converges to $0$. 
Hence we have
$$\lim_{\lambda\to0}\lambda T_\lambda{\sf e}
= \lim_{\lambda\to 0}[(\lambda v_\lambda + \lambda v_0)|_{\Gamma_*}+ a_0]= a_0 .$$
This then implies
\begin{equation*}
\lim_{\lambda\to0} (B_\lambda {\sf e}|{\sf e})_{\Gamma_*}
 = l_*^2u_* \frac{|\Gamma_*|^2}{(\kappa_*|1)_{\Omega}+\kappa_{\Gamma*}|\Gamma_*|}-
\,\frac{(n-1)\sigma_*|\Gamma_*|}{R_*^2}<0,
\end{equation*}
 if the stability condition
does not hold, i.e.\ if $\zeta_\ast>1$. Therefore also in this case $B_\lambda$ is not positive semi-definite for small $\lambda>0$.

\medskip

\noindent (c) \, Next we consider the behavior of
$(B_\lambda \rho|\rho)_{\Gamma_*}$ as $\lambda\to\infty$. We intend to show that
$B_\lambda$ is positive definite for large $\lambda$.  We have
\begin{align*}
\lambda T_\lambda&= \lambda(\kappa_{\Gamma*}\lambda -d_{\Gamma*}\Delta_\Gamma +D_\lambda)^{-1}
\to 1/\kappa_{\Gamma*}\quad \mbox{ for }\; \lambda\to\infty,
\end{align*}
 as $D_\lambda$ is of lower order, by part (b) of Lemma 4.5.
This implies for a given $g\in D(A_*)$
\begin{align*}
(B_\lambda g|g)_{\Gamma_*}&= l_*^2u_*\lambda (T_\lambda g|g)_{\Gamma_*}+ \sigma_*(A_*g|g)_{\Gamma_*} + \gamma_*\lambda|g|^2_{\Gamma_*}\\
&\geq (\gamma_*\lambda -\frac{(n-1)\sigma_*}{R_*^2})|g|^2_{\Gamma_*}+l_*^2u_*\lambda (T_\lambda g|g)_{\Gamma_*}\\
&\sim (\gamma_*\lambda -\frac{(n-1)\sigma_*}{R_*^2}+ \frac{l_*^2u_*}{\kappa_{\Gamma*}})|g|^2_{\Gamma_*},
\end{align*}
as $\lambda\to\infty$.
We have thus shown that  $B_\lambda$ is positive definite if
$\gamma_*>0$ and $\lambda> (n-1)\sigma_*/\gamma_*R_*^2$,
or if
\begin{equation}
\label{str-cond}
\gamma_*=0\quad\text{and}\quad
{l_*^2u_*}/{\kappa_{\Gamma*}} > {(n-1)\sigma_*}/{R_*^2}.
\end{equation}
In particular, for $\gamma_*=0$ and small $l_*^2$ the latter condition condition will be violated, in general.

\medskip

\noindent (d) \, In summary, concentrating on the cases $\gamma*>0$ or \eqref{str-cond},
we have shown that $B_\lambda$ is
not positive semi-definite for small $\lambda>0$ if either
$\Gamma_*$ is not connected or the stability condition does not
hold, and $B_\lambda$ is always positive definite for large
$\lambda$.
Let
\begin{equation*}
\lambda_0 = \sup\{\lambda>0:\, B_\mu \mbox{ is not positive semi-definite for each } \mu\in(0,\lambda]\}.
\end{equation*}
Since $B_\lambda$ has compact resolvent, $B_\lambda$ has a negative eigenvalue for each $\lambda<\lambda_0$. This implies that $0$ is an eigenvalue of
$B_{\lambda_0}$, thereby proving that $-L$ admits the positive eigenvalue $\lambda_0$.

Moreover, we have also shown that
\begin{equation*}
 B_0\rho= \lim_{\lambda\to 0}[ l^2_* u_*\lambda T_\lambda \rho +\gamma_*\lambda\rho
 +\sigma_* A_* \rho]
= \frac{l_*^2 u_*|\Gamma_*|}{(\kappa_*|1)_{\Omega}+\kappa_{\Gamma*}|\Gamma_*|} P_0 \rho +\sigma_* A_* \rho ,
\end{equation*}
where $P_0\rho:=(I-Q_0)\rho=(\rho|{\sf e})_{\Gamma_*}/|\Gamma_*|$.
Therefore, $B_0$ has the eigenvalue
\begin{equation*}
\frac{u_*l_*^2|\Gamma_*|}{(\kappa_*|1)_\Omega+\kappa_{\Gamma*}|\Gamma_*|}-\frac{(n-1)\sigma_*}{R_*^2}
=\frac{u_*l_*^2|\Gamma_*|}{(\kappa_*|1)_\Omega+\kappa_{\Gamma*}|\Gamma_*|}
(1-\zeta_\ast)
\end{equation*}
with eigenfunction ${\sf e}$, and in case $m>1$ it also has the
eigenvalue $-(n-1)\sigma_*/R_*^2$ with precisely $m-1$ linearly
independent eigenfunctions of the form $\sum_k a_k {\sf e}_k$ with
$\sum_k a_k=0$. 

As $\lambda$ varies from $0$ to $\lambda_0$,
all the negative eigenvalues of $B_0$ identified above
will eventually have to cross $0$ along the real axis.
At each of these occasions, $-L$
will inherit at least one positive eigenvalue, which will then remain positive.
This implies that $-L$ has exactly $m$
positive eigenvalues if the stability condition does not hold, and
$m-1$ otherwise. This covers the case $\gamma_*>0$ as well as \eqref{str-cond}.

\medskip

\noindent
(e) \, To cover the remaining we assume $\gamma_*=0$ and $\kappa_{\Gamma*}(n-1)/R_*^2>u_*l_*^2/\sigma_*=:\delta_*$.
Suppose $\lambda>0$ is an eigenvalue of $L_0$. Then there is $\rho\neq0$ such that
$$ (\lambda\kappa_{\Gamma*} -d_{\Gamma*}\Delta_*+D_\lambda)A_*\rho +\lambda \delta_*\rho=0.$$
Multiplying this equation in $L_2(\Gamma_*)$ by $A_*\rho$ and integrating by parts one obtains the identity
$$\lambda\kappa_{\Gamma*}|A_*\rho|_{\Gamma_*}^2 + d_{\Gamma*}|\nabla_{\Gamma_*}A_*\rho|^2_{\Gamma_*}+(D_\lambda A_*\rho|A_*\rho)_{\Gamma_*} +\lambda\delta_*(A_*\rho|\rho)_{\Gamma_*}=0.$$
As $D_\lambda$ is positive definite in $L_2(\Gamma_*)$ this equation implies
$$\lambda\kappa_{\Gamma*}|A_*\rho|_{\Gamma_*}^2 +\lambda\delta_*(A_*\rho|\rho)_{\Gamma_*}\leq0.$$
\goodbreak
Let $P$ denote the projection onto the kernel $\cN(\Delta_*)$ and $Q=I-P$. Since $P,Q$ commute with $A_*$ this implies
$$\lambda\kappa_{\Gamma*}|A_*Q\rho|_{\Gamma_*}^2 +\lambda\kappa_{\Gamma*}|A_*P\rho|_{\Gamma_*}^2+\lambda\delta_*(A_*P\rho|P\rho)_{\Gamma_*}\leq0,$$
as $A_*$ is positive semi-definite on $\cR(Q)=\cR(\Delta_*)$.
Now  $A_*P=-((n-1)/R_*^2)P$ and
$$0\geq\lambda\kappa_{\Gamma*}|A_*P\rho|_{\Gamma_*}^2+\lambda\delta_*(A_*P\rho|P\rho)_{\Gamma_*}=
\lambda\frac{n-1}{R_*^2}\Big[\kappa_{\Gamma*}\frac{n-1}{R_*^2}-\delta_*\Big]|P\rho|_{\Gamma_*}^2\geq0,$$
hence $P\rho=0$ and $A_*Q\rho=0$. This implies $A_*\rho=0$ and therefore $\rho=0$ as $\delta_*>0$. This shows that there are no positive eigenvalues
of $L_0$ in case $\gamma_*=0$ and $\kappa_{\Gamma*}(n-1)/R_*^2>u_*l_*^2/\sigma_*$. This completes the proof.

\section{The Semiflow in Presence of Kinetic Undercooling}
In this section we assume throughout $\gamma(s)>0$ for all $0<s<u_c$, i.e.\ kinetic undercooling is present
at the relevant temperature range. In this case we may apply the results in \cite{PSZ09} and \cite{KPW10}, resulting in a rather complete analysis of the problem.

\subsection{Local Well-Posedness}

\medskip

\noindent
To prove local well-posedness we employ the direct mapping method as introduced in Section 3. As base space we use
\begin{equation*}
 X_0= L_p(\Omega)\times W^{-1/p}_p(\Sigma)\times W^{1-1/p}_p(\Sigma),
\end{equation*}
and we set
\begin{equation*}
\begin{aligned}
X_1&=\big\{(v,v_\Gamma,\rho)\in H^2_p(\Omega\setminus\Sigma)\times W^{2-1/p}_p(\Sigma)
\times  W^{3-1/p}_p(\Sigma):\\
& \hspace{5cm} [\![v]\!]=0,\; v_\Gamma=v_{|_\Sigma},\; \partial_\nu v_{|_{\partial\Omega}}=0\big\}.
\end{aligned}
\end{equation*}
The trace space $X_\gamma$ then becomes for $p>n+2$
\begin{equation*}
\begin{aligned}
X_{\gamma} &=\big\{(v,v_\Gamma,\rho)\in W^ {2-2/p}_p(\Omega\setminus\Sigma)\times W^{2-3/p}_p(\Sigma)\times W^{3-3/p}_p(\Sigma):\\
& \hspace{5cm} [\![v]\!]=0,\; v_\Gamma=v_{|_\Sigma},\; \partial_\nu v_{|_{\partial\Omega}}=0\big\},
\end{aligned}
\end{equation*}
and that with the time weight $t^{1-\mu}$, $1\geq\mu>1/p$,
\begin{equation*}
\begin{aligned}
X_{\gamma,\mu}&=\big\{(v,v_\Gamma,\rho)\in W^ {2\mu-2/p}_p(\Omega\setminus\Sigma)\times W^{2\mu-3/p}_p(\Sigma)\times W^{2\mu+1-3/p}_p(\Sigma):\\
&\hspace{5cm} [\![v]\!]=0,\; v_\Gamma=v_{|_\Sigma},\; \partial_\nu v_{|_{\partial\Omega}}=0\big\},
\end{aligned}
\end{equation*}
Note that
\begin{equation}
\label{crucialembedding}
X_{\gamma,\mu}\hookrightarrow BUC^1(\Omega\setminus\Sigma)\times C^1(\Sigma)\times C^2(\Sigma),\end{equation}
provided $2\mu>1 +(n+2)/p$, which is feasible as $p>n+2$. In the sequel, we only consider this range of $\mu$.
We want to rewrite system \eqref{transformed} abstractly as the quasilinear problem in $X_0$
\begin{align}\label{abstract}
\dot{z}+A(z)z&= F(z),\quad z(0)=z_0,
\end{align}
where $z=(v,v_\Gamma,\rho)$ and $z_0=(v_0,v_{\Gamma0},\rho_0)$. Here the quasilinear part $A(z)$ is the diagonal matrix operator defined by
\begin{equation*}
\label{A(z)}
-A(z)={\rm diag}
\left[\begin{aligned}
&(d(v)/\kappa(v))(\Delta - M_2(\rho):\nabla^2)\\
&(d_\Gamma(v_\Gamma)/\kappa_\Gamma(v_\Gamma))(P_\Gamma(\rho)M_0(\rho))^2:\nabla_\Sigma^2\; \\
&(\sigma(v_\Gamma)/\gamma(v_\Gamma))\cG(\rho):\nabla_\Sigma^2
 \end{aligned}\right]
\end{equation*}
with
$M_2(\rho)=M_1(\rho)+M_1^{\sf T}(\rho) -M_1(\rho)M_1^{\sf T}(\rho)$.
The semilinear part $F(z)$ is given by
\begin{equation*}
\left[
\begin{aligned}
&\cR(\rho)v+\frac{1}{\kappa(v)}\big\{d^\prime(v)|(I-M_1(\rho))\nabla v\big|^2
 - d(v)((I-M_1(\rho)):\nabla M_1(\rho)|\nabla v)\big\}\\
&\frac{1}{\kappa_\Gamma(v_\Gamma)}\big\{-\cB(v_\Gamma,\rho)v
-[l(v_\Gamma) +l_\Gamma(v_\Gamma)\cH(\rho)-\gamma(v_\Gamma)\beta(\rho)\partial_t\rho]
\beta(\rho)\partial_t\rho + m_3\big\} \\
&\varphi(v_\Gamma) /\beta(\rho)\gamma(v_\Gamma) + \sigma(v_\Gamma)\cF(\rho)/\gamma(v_\Gamma)
\end{aligned}\right]
\end{equation*}
where  $\varphi(s)=[\![\psi(s)]\!]$ and
\begin{equation*}
m_3=
-d_\Gamma(v_\Gamma)(P_\Gamma(\rho)M_0(\rho))^2:\nabla_\Sigma^2v_\Gamma
-\cC(v_\Gamma,\rho)v_\Gamma.
\end{equation*}
We note that $m_3$ depends  on $v_\Gamma$, $\nabla_\Sigma v_\Gamma$,
and on $\rho$, $\nabla_\Sigma\rho$, $\nabla^2_\Sigma\rho$, but not on $\nabla_\Sigma^2v_\Gamma$,
hence is of lower order.
Apparently, the first two components of $F(z)$ contain the time derivative $\partial_t\rho$; we may replace it by
$$\partial_t\rho = \{\varphi(v_\Gamma)+\sigma(v_\Gamma)\cH(\rho)\}/\beta(\rho)\gamma(v_\Gamma),$$
to see that it is of lower order as well.

Now fix a ball $\BB:=B_{X_{\gamma,\mu}}(z_0,R)\subset X_{\gamma,\mu}$,
where $|\rho_0|_{C^1(\Sigma)}\leq\eta$ for some sufficiently small $\eta>0$.
Then it is not difficult to verify that
\begin{equation*}
(A,F)\in C^1(\BB,\cB(X_1,X_0)\times X_0)
\end{equation*}
provided
$d_i,\psi_i,d_\Gamma,\sigma,\gamma\in C^3(0,\infty)$ and
$d_j,\kappa_j,\sigma,\gamma>0$ on $(0,u_c)$, $j=1,2,\Gamma$,
and provided
$2\geq 2\mu> 1+ (n+2)/p$ as before.
Moreover, as $A(z)$ is diagonal, well-known results about elliptic differential operators show that $A(z)$ has the property of maximal regularity of type $L_{p}$, and also of type $L_{p,\mu}$, for each $z\in\BB$.
In fact, for small $\eta>0$ and $R>0$, $A(z)$ is small perturbation of
\begin{equation*}
A_\#(z)={\rm diag}\big[-(d(v)/\kappa(v))\Delta, -(d_\Gamma(v_\Gamma)/\kappa_\Gamma(v_\Gamma))\Delta_\Sigma,-(\sigma(v_\Gamma)/\gamma(v_\Gamma))\Delta_\Sigma \big].
\end{equation*}
Therefore we may apply \cite[Theorem 2.1]{KPW10} to obtain local well-posedness of \eqref{abstract}, i.e.\ a unique local solution
$$z\in H^1_{p,\mu}((0,a);X_0)\cap L_{p,\mu}((0,a);X_1)\hookrightarrow C([0,a];X_{\gamma,\mu})\cap C((0,a];X_\gamma)$$
which depends continuously on the initial value $z_0\in \BB$. The resulting solution map $[z_0\mapsto z(t)]$ defines a local semiflow in $X_{\gamma ,\mu}$.

\subsection{Nonlinear Stability of Equilibria}

\medskip

\noindent
Let $e_*=(u_*,u_{\Gamma*},\Gamma_*)$ denote an equilibrium as in Section 4. In this case we choose $\Sigma=\Gamma_*$ as a 
reference manifold, and
as shown in the previous subsection we obtain the abstract quasilinear parabolic problem
\begin{align}\label{abstract*}
\dot{z}+A(z)z&= F(z),\quad z(0)=z_0,
\end{align}
with $X_0$, $X_1$, $X_\gamma$ as above. We set $z_*=(u_*,u_{\Gamma*},0)$. Assuming that
$\zeta_\ast\neq 1$ in the stability condition, we have shown in Section 4 that the equilibrium $z_*$ is normally hyperbolic.
Therefore we may apply \cite[Theorems 2.1 and 6.1 ]{PSZ09} to obtain the following result.
\goodbreak
\begin{theorem}
\label{nonlinstab}
Let $p>n+2$. Suppose $\gamma>0$ on $(0,u_c)$ and the assumptions of \eqref{regularity-coeff} hold true.
As above $\cE$ denotes the set of equilibria of \eqref{abstract*}, and we fix some $z_*\in\cE$.
Then we have
\vspace{1mm}\\
{\bf (a)} If $\Gamma_*$ is connected and $\zeta_*<1$ then $z_*$ is stable in $X_\gamma$, and there exists $\delta>0$ such
that the unique solution $z(t)$ of (\ref{abstract*}) with initial
value $z_0\in X_\gamma$ satisfying $|z_0-z_*|_{\gamma}<\delta$
exists on $\R_+$ and converges at an exponential rate in $X_\gamma$
to some $z_\infty\in\cE$ as $t\rightarrow\infty$.
\vspace{1mm}
\\
{\bf (b)} If $\Gamma_*$ is disconnected or if $\zeta_*>1$ then $z_*$ is unstable in $X_\gamma$ and even in $X_0$.
For each sufficiently small $\rho>0$ there is $\delta\in(0,\rho]$ such that the solution $z(t)$ of \eqref{abstract*} with initial
value $z_0\in X_\gamma$ subject to $|z_0-z_*|_{\gamma}<\delta$ either satisfies
\begin{itemize}
\item[(i)] ${\rm dist}_{X_\gamma}(z(t_0);\cE)>\rho$ for some finite time $t_0>0$; or
\vspace{1mm}
\item[(ii)] $z(t)$ exists on $\R_+$ and converges at exponential rate in $X_\gamma$ to some $z_\infty\in\cE$.
\end{itemize}
\end{theorem}
\begin{remark}
The only equilibria which are excluded from our analysis are those with $\zeta_*=1$, which means ${\sf E}^\prime_e(u_*)=0$. These are critical points of the function ${\sf E}_e(u)$ at which a bifurcation may occur. In fact, if such $u_*$ is a maximum or a minimum of ${\sf E}_e$ then two branches of $\cE$ meet at $u_*$, a stable and and an unstable one, which means that $(u_*,\Gamma_*)$ is a turning point in $\cE$.
\end{remark}
\subsection{The Local Semiflow on the State Manifold}

\medskip

\noindent
Here we follow the approach introduced in \cite{KPW11} for the two-phase Navier-Stokes problem and in \cite{PSZ10} for the two-phase Stefan problem,
see also \cite{KPW10} for the Mullins-Sekerka problem.

We denote by $\cMH^2(\Omega)$ the closed
$C^2$-hypersurfaces contained in $\Omega$. It can be shown that
$\cMH^2(\Omega)$ is a $C^2$-manifold: the charts are the
parameterizations over a given hypersurface $\Sigma$ according to
Section 3, and the tangent space consists of the normal vector
fields on $\Sigma$. We define a metric on $\cMH^2(\Omega)$ by means
of
$$d_{\cMH^2}(\Sigma_1,\Sigma_2):= d_H(\cN^2\Sigma_1,\cN^2\Sigma_2),$$
where $d_H$ denotes the Hausdorff metric on the compact subsets of $\R^n$ introduced in Section 2.
This way $\cMH^2(\Omega)$ becomes a Banach manifold of class $C^2$.

Let $d_\Sigma(x)$ denote the signed distance for $\Sigma$ as in Section 2.
We may then define the {\em canonical level function} $\varphi_\Sigma$ by means of
$$\varphi_\Sigma(x) = \phi(d_\Sigma(x)),\quad x\in\R^n,$$
where
$$\phi(s)=s \chi(s/a) + (1-\chi(s/a))\,{\rm sgn}\, s ,\quad s\in \R.$$
Then it is easy to see that $\Sigma=\varphi_\Sigma^{-1}(0)$, and
$\nabla \varphi_\Sigma(x)=\nu_\Sigma(x)$, for $x\in \Sigma$.
Moreover, $0$ is an eigenvalue of $\nabla^2\varphi_\Sigma(x)$, and
the remaining eigenvalues of $\nabla^2\varphi_\Sigma(x)$ are the
principal curvatures of $\Sigma$ at $x\in\Sigma$.

If we consider the subset $\cMH^2(\Omega,r)$ of $\cMH^2(\Omega)$ which consists of all closed
hypersurfaces $\Gamma\in \cMH^2(\Omega)$ such that $\Gamma\subset \Omega$ satisfies a
(interior and exterior) ball condition with fixed radius $r>0$, then the map
\begin{equation}
\Upsilon:\cMH^2(\Omega,r)\to C^2(\bar{\Omega}),\quad
\Upsilon(\Gamma):=\varphi_\Gamma,
\end{equation}
is an isomorphism of the metric space $\cMH^2(\Omega,r)$ onto
$\Upsilon(\cMH^2(\Omega,r))\subset C^2(\bar{\Omega})$.
\medskip\\
Let $s-(n-1)/p>2$. Then we define
\begin{equation}
\label{definition-W-r}
W^s_p(\Omega,r):=\{\Gamma\in\cMH^2(\Omega,r): \varphi_\Gamma\in W^s_p(\Omega)\}.
\end{equation}
In this case the local charts for $\Gamma$ can be chosen of class
$W^s_p$ as well. A subset $A\subset W^s_p(\Omega,r)$ is said to be
(relatively) compact, if $\Upsilon(A)\subset W^s_p(\Omega)$ is
(relatively) compact.

As an ambient space for the
state manifold of \eqref{stefan-var} we consider
the product space $C(\bar{G})\times \cMH^2$, due to continuity of temperature
and curvature.

We define the state manifold $\cSM$ for  \eqref{stefan-var} as follows:
\begin{equation}
\label{phasemanifg}
\begin{aligned}
\cSM:=\big\{(u,\Gamma)\in C(\bar{\Omega})\times \cMH^2:
        & \;u\in W^{2-2/p}_p(\Omega\setminus\Gamma),\, \Gamma\in W^{3-3/p}_p ,\\
&\; 0<u<u_c \mbox{ in } \bar{\Omega},\; \partial_\nu u=0 \mbox{ on } \partial\Omega\big\}.
\end{aligned}
\end{equation}
Charts for this manifold are obtained by the charts induced by $\cMH^2(\Omega)$ followed by a Hanzawa transformation
as in Section~3.
Note that there is  no need to incorporate the dummy variable $u_\Gamma$ into the definition of the state manifold, as $u_\Gamma=u|_{\Gamma}$ whenever $u_\Gamma$ appears.

Applying the result in subsection 5.1
and re-parameterizing the interface repeatedly, we see that
(\ref{stefan-var}) yields a local semiflow on $\cSM$.

\begin{theorem}
Let $p>n+2$. Suppose $\gamma>0$ on $(0,u_c)$ and the assumptions of \eqref{regularity-coeff} hold true.

Then problem (\ref{stefan-var}) generates a local semiflow
on the state manifold $\cSM$. Each solution $(u,\Gamma)$ exists on a maximal time
interval $[0,t_*)$, where $t_*=t_*(u_0,\Gamma_0)$.
\end{theorem}
\medskip

\subsection{Global Existence and Convergence}
There are several
obstructions to global existence for the Stefan problem with variable surface tension
\eqref{stefan-var}:
\begin{itemize}
\item {\em regularity}: the norms of $u(t)$ or $\Gamma(t)$ become unbounded;
\item {\em well-posedness}: the temperature may reach $0$ or $u_c$;
\item {\em geometry}: the topology of the interface changes;\\
    or the interface touches the boundary of $\Omega$;\\
    or the interface contracts to a point.
\end{itemize}
Let $(u,\Gamma)$ be a solution in the state manifold $\cSM$.
By a {\em uniform ball condition} we mean the existence of a radius
$r_0>0$ such that for each $t$, at each point $x\in\Gamma(t)$ there
exist centers $x_i\in \Omega_i(t)$ such that $B_{r_0}(x_i)\subset
\Omega_i$ and $\Gamma(t)\cap \bar{B}_{r_0}(x_i)=\{x\}$, $i=1,2$.
Note that this condition bounds the curvature of $\Gamma(t)$,
prevents it from shrinking to a point, from touching the outer
boundary $\partial \Omega$, and from undergoing topological changes.

With this property, combining the semiflow for (\ref{stefan-var}) with
the Lyapunov functional and compactness we obtain the following
result.
\begin{theorem}
\label{Qual}
Let $p>n+2$.
Suppose $\gamma>0$ on $(0,u_c)$ and the assumptions of \eqref{regularity-coeff} hold true.
Suppose that $(u,\Gamma)$ is a solution of
(\ref{stefan-var}) in the state manifold $\cSM$ on its maximal time interval $[0,t_*)$.
Assume the following on $[0,t_*)$: there is a constant $M>0$ such that
\begin{itemize}
\item[(i)] $|u(t)|_{W^{2-2/p}_p}+|\Gamma(t)|_{W^{3-3/p}_p}\leq M<\infty$;
\vspace{1mm}
\item[(ii)] $0<1/M\leq u(t)  \le u_c-1/M$;
\vspace{1mm}
\item[(iii)] $\Gamma(t)$ satisfies a uniform ball condition.
\end{itemize}
Then $t_*=\infty$, i.e.\ the solution exists globally, and
it converges in $\cSM$
to some equilibrium $(u_\infty,\Gamma_\infty)\in\cE$. On the contrary, if $(u(t),\Gamma(t))$ is a global solution in $\cSM$ which converges to an equilibrium $(u_*,\Gamma_*)$ in $\cSM$ as $t\to\infty$, then properties $(i)$-$(iii)$ are valid.
\end{theorem}
\begin{proof}
Assume that assertions (i)--(iii) are valid. Then
$\Gamma([0,t_*))\subset W^{3-3/p}_p(\Omega,r)$ is bounded, hence
relatively compact in $W^{3-3/p-\ve}_p(\Omega,r)$.
Thus we may cover this set by finitely many balls with centers $\Sigma_k$ real analytic in such a way that
${\rm dist}_{W^{3-3/p-\ve}_p}(\Gamma(t),\Sigma_j)\leq \delta$ for some $j=j(t)$, $t\in[0,t_*)$. Let $J_k=\{t\in[0,t_*):\, j(t)=k\}$. Using for each $k$ a Hanzawa-transformation $\Xi_k$, we see that the pull backs $\{u(t,\cdot)\circ\Xi_k:\, t\in J_k\}$ are bounded in $W^{2-2/p}_p(\Omega\setminus \Sigma_k)$, hence relatively compact in $W^{2-2/p-\ve}_p(\Omega\setminus\Sigma_k)$. Employing now the results in subsection 5.1  we obtain solutions
 $(u^1,\Gamma^1)$ with initial configurations $(u(t),\Gamma(t))$ in the state manifold on a common
time interval, say $(0,\tau]$, and by uniqueness we have
$$(u^1(\tau),\Gamma^1(\tau))=(u(t+\tau),\Gamma(t+\tau)).$$
Continuous dependence implies then relative compactness of
$(u(\cdot),\Gamma(\cdot))$ in $\cSM$. In particular,
$t_*=\infty$ and the orbit $(u,\Gamma)(\R_+)\subset\cSM$ is
relatively compact. The negative total entropy is a strict Lyapunov
functional, hence the limit set $\omega(u,\Gamma)\subset
\cSM$ of a solution is contained in the set $\cE$ of
equilibria.  By compactness $\omega(u,\Gamma)\subset \cSM$ is
non-empty, hence the solution comes close to $\cE$, and stays there.
Then we may apply the convergence result Theorem \ref{nonlinstab}.
The converse is proved by a compactness argument.
\end{proof}

\section{The Semiflow without Kinetic Undercooling}
In this section we assume throughout $\gamma(s)=0$ for all $s>0$, i.e\ kinetic undercooling is absent. In this case we may apply the results in \cite{PSZ09} and \cite{KPW10} too, but we have to work harder to apply them. 
At first we prove \eqref{mcflow} as follows.
According to \eqref{T-Gamma-V} we know that 
\begin{equation*}
 T_\Gamma(u_\Gamma)V:=(\omega_\Gamma(u_\Gamma)-\cH^\prime(\Gamma))V
=\frac{\lambda^\prime(u_\Gamma)}{\kappa_\Gamma(u_\Gamma)}
\big\{{\rm div}_\Gamma (d_\Gamma(u_\Gamma)\nabla_\Gamma u_\Gamma) + [\![d(u)\partial_\nu u]\!]\big\}.
\end{equation*}
Next we observe
\begin{align*}
\frac{\lambda^\prime(u_\Gamma)}{\kappa_\Gamma(u_\Gamma)}& {\rm div}_\Gamma( d_\Gamma(u_\Gamma)\nabla_\Gamma u_\Gamma)\\
&= \frac{1}{\kappa_\Gamma(u_\Gamma)}{\rm div}_\Gamma (d_\Gamma(u_\Gamma)\nabla_\Gamma \lambda(u_\Gamma))- \frac{d_\Gamma(u_\Gamma)}{\kappa_\Gamma(u_\Gamma)}\lambda^{\prime\prime}(u_\Gamma)|\nabla_\Gamma u_\Gamma|^2\\
& = {\rm div}_\Gamma \Big(\frac{d_\Gamma(u_\Gamma)}{\kappa_\Gamma(u_\Gamma)}\nabla_\Gamma \lambda(u_\Gamma)\Big)-
\frac{d_\Gamma(u_\Gamma)}{\kappa_\Gamma(u_\Gamma)}
\Big\{\lambda^{\prime\prime}(u_\Gamma)-\lambda^\prime(u_\Gamma)\frac{\kappa^\prime_\Gamma(u_\Gamma)}{\kappa_\Gamma(u_\Gamma)}\Big\}
|\nabla_\Gamma u_\Gamma|^2\\
&= \Delta_\Gamma h_\Gamma(u_\Gamma)
-\frac{d_\Gamma(u_\Gamma)}{\kappa_\Gamma(u_\Gamma)}
\Big\{\lambda^{\prime\prime}(u_\Gamma)-\lambda^\prime(u_\Gamma)\frac{\kappa^\prime_\Gamma(u_\Gamma)}{\kappa_\Gamma(u_\Gamma)}\Big\}
|\nabla_\Gamma u_\Gamma|^2\\\end{align*}
where $h_\Gamma$ denotes the antiderivative of $d_\Gamma \lambda^\prime/\kappa_\Gamma$ with $h_\Gamma(u_m)=0$.
We note that by a partial integration
$$
h_\Gamma(s)= \lambda(s) \frac{d_\Gamma(s)}{\kappa_\Gamma(s)} -\int_{u_m}^s \lambda(\tau) (\frac{d_\Gamma}{\kappa_\Gamma})^\prime(\tau)d\tau=: \lambda(s) \frac{d_\Gamma(s)}{\kappa_\Gamma(s)}-f_\Gamma(s).
$$
Now employing $\lambda(u_\Gamma)=-\cH(\Gamma)$ leads to the identity
\begin{align*}
&T_\Gamma(u_\Gamma)\{V-\frac{d_\Gamma(u_\Gamma)}{\kappa_\Gamma(u_\Gamma)} \cH(\Gamma)-f_\Gamma(u_\Gamma)\}\\
&=\frac{ \lambda^\prime(u_\Gamma)}{\kappa_\Gamma(u_\Gamma)}[\![d(u)\partial_\nu u]\!]
-\frac{d_\Gamma(u_\Gamma)}{\kappa_\Gamma(u_\Gamma)}\{\lambda^{\prime\prime}(u_\Gamma)-\lambda^\prime(u_\Gamma)\frac{\kappa_\Gamma^\prime(u_\Gamma)}{\kappa_\Gamma(u_\Gamma)}\}|\nabla_\Gamma u_\Gamma|^2\\
&
+[\,\omega_\Gamma(u_\Gamma) -{\rm tr}L_\Gamma^2\,]h_\Gamma(u_\Gamma),
\end{align*}
hence applying the inverse of $T_\Gamma(u_\Gamma)$ we arrive at
\begin{equation}
\label{mcflow-2}
\kappa_\Gamma(u_\Gamma)V -d_\Gamma(u_\Gamma)\cH(\Gamma) = \kappa_\Gamma(u_\Gamma)\{f_\Gamma(u_\Gamma) + F_\Gamma(u,u_\Gamma)\},
\end{equation}
where
\begin{align*}
F_\Gamma(u,u_\Gamma)&=[\kappa_\Gamma(u_\Gamma)T_\Gamma(u_\Gamma)]^{-1}
\big\{ \lambda^\prime(u_\Gamma)[\![d(u)\partial_\nu u]\!]\\
&-d_\Gamma(u_\Gamma)
[(\lambda^{\prime\prime}(u_\Gamma)-\lambda^\prime(u_\Gamma)\kappa_\Gamma^\prime(u_\Gamma)/\kappa_\Gamma(u_\Gamma)]|\nabla_\Gamma u_\Gamma|^2\\
&+\kappa_\Gamma(u_\Gamma) [\,\omega_\Gamma(u_\Gamma) -{\rm tr}L_\Gamma^2\,]h_\Gamma(u_\Gamma)\big\}.
\end{align*}
In the sequel we will replace the Gibbs-Thomson law by the dynamic equation \eqref{mcflow-2} plus the compatibility condition $\varphi(u_{\Gamma0})+\sigma(u_{\Gamma0})\cH(\Gamma_0)=0$
at time $t=0$.

\subsection{Local Well-Posedness}

\medskip

\noindent
To prove local well-posedness we employ the direct mapping method as introduced in Section 3. As base space we use as in Section 5
$$ X_0= L_p(\Omega)\times W^{-1/p}_p(\Sigma)\times W^{1-1/p}_p(\Sigma),$$
and we let $X_1$, $X_\gamma$ and  $X_{\gamma,\mu}$ as defined there.

We rewrite system \eqref{transformed} abstractly as the quasilinear problem in $X_0$
\begin{align}\label{abstract0}
\dot{z}+A_0(z)z&= F_0(z),\quad z(0)=z_0,
\end{align}
where $z=(v,v_\Gamma,\rho)$ and $z_0=(v_0,v_{\Gamma0},\rho_0)$. Here the quasilinear part $A_0(z)$
is the diagonal matrix operator defined by
\begin{equation*}
\label{A(z)0}
-A_0(z)={\rm diag}
\left[\begin{aligned}
&(d(v)/\kappa(v))(\Delta - M_2(\rho):\nabla^2) \\
&(d_\Gamma(v_\Gamma)/\kappa_\Gamma(v_\Gamma))(P_\Gamma(\rho)M_0(\rho))^2:\nabla_\Sigma^2 \\
&(d_\Gamma(v_\Gamma)/\kappa_\Gamma(v_\Gamma))\cG(\rho):\nabla_\Sigma^2
 \end{aligned}\right]
\end{equation*}
with
$M_2(\rho)=M_1(\rho)+M_1^{\sf T}(\rho) -M_1(\rho)M_1^{\sf T}(\rho)$.
The semilinear part $F_0(z)$ is given by
\begin{equation*}
\label{F(z)0}
\left[
\begin{aligned}
&\cR(\rho)v+\frac{1}{\kappa(v)}\big\{d^\prime(v)|(I-M_1(\rho))\nabla v\big|^2
 - d(v)((I-M_1(\rho)):\nabla M_1(\rho)|\nabla v)\big\}\\
&\frac{1}{\kappa_\Gamma(v_\Gamma)}
\big\{-\cB(v_\Gamma,\rho)v
-[l(v_\Gamma) +l_\Gamma(v_\Gamma)\cH(\rho)]
\beta(\rho)\partial_t\rho + m_3 \big\} \\
&(d_\Gamma(v_\Gamma) /\kappa_\Gamma(v_\Gamma))\cF(\rho)
+ \big\{f_\Gamma(v_\Gamma)+ F_\Gamma(v,v_\Gamma,\rho)\big\}/\beta(\rho)
\end{aligned}
\right]
\end{equation*}
where by abuse of notation $F_\Gamma$ here means the transformed $F_\Gamma$ introduced previously,
and where
\begin{equation*}
m_3=-d_\Gamma(v_\Gamma)(P_\Gamma(\rho)M_0(\rho))^2:\nabla_\Sigma^2v_\Gamma
-\cC(v_\Gamma,\rho)v_\Gamma.
\end{equation*}
Again, the first two components of $F_0(z)$ contain the time derivative $\partial_t\rho$. We replace it by
the transformed version of \eqref{mcflow-2}
\begin{equation*}
\partial_t\rho = \big\{f_\Gamma(v_\Gamma)+F_\Gamma(v,v_\Gamma,\rho)
+d_\Gamma(v_\Gamma)/\kappa_\Gamma(v_\Gamma)\cH(\rho)\big\}/\beta(\rho),
\end{equation*}
to see that it leads to a lower order term, as in Section 5.

Provided that
$T_{\Gamma_0}(v_{\Gamma0})$ is invertible we may proceed as in Section 5,
applying  Theorem 2.1 in \cite{KPW10}, to obtain local well-posedness, i.e.\ a unique local solution
$$z\in H^1_{p,\mu}((0,a);X_0)\cap L_{p,\mu}((0,a);X_1)\hookrightarrow C([0,a];X_{\gamma,\mu})\cap C((0,a];X_\gamma)$$
which depends continuously on the initial value $z_0\in \BB$.
The resulting solution map $[z_0\mapsto z(t)]$ defines a local semiflow in $X_{\gamma,\mu}$.

\subsection{Nonlinear Stability of Equilibria}

\medskip

\noindent
Let $e_*=(u_*,u_{\Gamma*},\Gamma_*)$ denote an equilibrium as in Section 4. In this case we choose $\Sigma=\Gamma_*$ as a 
reference manifold, and
as shown in the previous subsection
we obtain the abstract quasilinear parabolic problem
\begin{align}\label{abstract0*}
\dot{z}+A_0(z)z&= F_0(z),\quad z(0)=z_0,
\end{align}
with $X_{0}$, $X_{1}$, $X_{\gamma}$ as above. We set $z_*=(u_*,u_{\Gamma*},0)$. Assuming well-posedness and 
$\zeta_*\neq 1$ in the stability condition, we have shown in Section 4 that the equilibrium $e_*$ is normally hyperbolic. Therefore we may apply once more \cite{PSZ09}, Theorems 2.1 and 6.1 to obtain the following result.
\goodbreak
\begin{theorem}
\label{nonlinstab0}
Let $p>n+2$.
Suppose $\gamma\equiv 0$, $\sigma\in C^4(0,u_c)$,
and the assumptions of \eqref{regularity-coeff} hold true.
 As above $\cE$ denotes the set of equilibria of \eqref{abstract*}, and we fix some $z_*\in\cE$.
Assume that the well-posedness condition
\begin{equation}\label{wellp0}
l_*\neq 0\quad \mbox{ and } \quad  u_*l_*^2/\sigma_*\neq \kappa_{\Gamma*}(n-1)/R_*^2
\end{equation}
is satisfied. Then we have
\vspace{1mm}\\
{\bf (a)} If $\Gamma_*$ is connected and $\zeta_*<1$, or if $\kappa_{\Gamma*}(n-1)/R_*^2>u_*l_*^2/\sigma_*$ then $z_*$ is stable in $X_\gamma$, and there exists $\delta>0$ such
that the unique solution $z(t)$ of (\ref{abstract*}) with initial
value $z_0\in X_\gamma$ satisfying $|z_0-z_*|_{\gamma}<\delta$
exists on $\R_+$ and converges at an exponential rate in $X_\gamma$
to some $z_\infty\in\cE$ as $t\rightarrow\infty$.
\vspace{1mm}\\
{\bf (b)} If $\kappa_{\Gamma*}(n-1)/R_*^2<u_*l_*^2/\sigma_*$, and  if $\Gamma_*$ is disconnected or if $\zeta_*>1$ then $z_*$ is unstable in $X_\gamma$ and even in $X_0$.
For each sufficiently small $\rho>0$ there is $\delta\in(0,\rho]$ such that the solution $z(t)$ of \eqref{abstract*} with initial
value $z_0\in X_\gamma$ subject to $|z_0-z_*|_{\gamma}<\delta$ either satisfies
\begin{itemize}
\item[(i)] ${\rm dist}_{X_\gamma}(z(t_0);\cE)>\rho$ for some finite time $t_0>0$; or
\vspace{1mm}
\item[(ii)] $z(t)$ exists on $\R_+$ and converges at exponential rate in $X_\gamma$ to some $z_\infty\in\cE$.
\end{itemize}
\end{theorem}

Thus the only cases which are excluded are $\zeta_*=1$, and the two values where the well-posedness condition \eqref{wellp0} is violated.

\subsection{The Local Semiflow on the State Manifold}

\medskip

\noindent
We define the state manifolds $\cSM_0$ for  \eqref{stefan-var} in case $\gamma\equiv0$ as follows.
\begin{equation}
\label{phasemanif0}
\begin{aligned}
\cSM_0:=\big\{(&u,\Gamma)\in C(\bar{\Omega})\times \cMH^2:
 u\in W^{2-2/p}_p(\Omega\setminus\Gamma),\, \Gamma\in W^{3-3/p}_p ,\\
&
0<u<u_c \mbox{ in } \bar{\Omega},\;  \partial_\nu u=0 \mbox{ on } \partial\Omega, \\
&\lambda(u_\Gamma)+\cH(\Gamma)=0\mbox{ on } \Gamma,\; T_\Gamma(u_\Gamma)\; \mbox{ is invertible in}\; L_2(\Gamma)\big\}.
\end{aligned}
\end{equation}
Charts for this manifold are obtained by the charts induced by $\cMH^2(\Omega)$
followed by a Hanzawa transformation as in Section~3.

Applying the result of subsection 6.1
and re-parameterizing the interface repeatedly, we see that
(\ref{stefan-var}) with $\gamma\equiv0$ yields a local semiflow on $\cSM_0$.

\begin{theorem}
Let $p>n+2$.
Suppose $\gamma\equiv 0$, $\sigma\in C^4(0,u_c)$,
and the assumptions of \eqref{regularity-coeff} hold true.

Then problem (\ref{stefan-var}) generates a local semiflow
on the state manifold $\cSM_0$. Each solution $(u,\Gamma)$ exists on a maximal time
interval $[0,t_*)$, where $t_*=t_*(u_0,\Gamma_0)$.
\end{theorem}

\medskip

\subsection{Global Existence and Convergence}

\medskip

In addition to the
obstructions to global existence for the Stefan problem with variable surface tension in the presence of kinetic undercooling there is an additional possibility for loss of well-posedeness:
\begin{itemize}
\item {\em regularity}: the norms of $u(t)$ or $\Gamma(t)$ become unbounded;
\item {\em well-posedness}: the temperature may reach $0$ or $u_c$; or\\ $T_\Gamma(u_\Gamma)$ may become non-invertible;
\item {\em geometry}: the topology of the interface changes;\\
    or the interface touches the boundary of $\Omega$;\\
    or the interface contracts to a point.
\end{itemize}
We set $\cE_0=\cSM_0\cap \cE$.
As in Section 5, combining the semiflow for (\ref{stefan-var}) with
the Lyapunov functional and compactness we obtain the following
result.
\begin{theorem}
\label{Qual0}
Let $p>n+2$.
Suppose $\gamma\equiv 0$, $\sigma\in C^4(0,u_c)$,
and the assumptions of \eqref{regularity-coeff} hold true.
Suppose that $(u,\Gamma)$ is a solution of
(\ref{stefan-var}) in the state manifold $\cSM_0$ on its maximal time interval $[0,t_*)$.
Assume the following on $[0,t_*)$: there is a constant $M>0$ such that
\begin{itemize}
\item[(i)] $|u(t)|_{W^{2-2/p}_p}+|\Gamma(t)|_{W^{3-3/p}_p}\leq M<\infty$;
\vspace{2mm}
\item[(ii)] $0<1/M\leq u(t)\leq u_c-1/M$;
\vspace{2mm}
\item[(iii)] $|\mu_j(t)|\geq 1/M$ holds for the eigenvalues of $T_{\Gamma(t)}(u_\Gamma)$;
\vspace{2mm}
\item[(iv)] $\Gamma(t)$ satisfies a uniform ball condition.
\end{itemize}
Then $t_*=\infty$, i.e.\ the solution exists globally, and  it converges in $\cSM_0$
to an equilibrium $(u_\infty,\Gamma_\infty)\in\cE_0$ .
Conversely, if $(u(t),\Gamma(t))$ is a global solution in $\cSM_0$ which converges to an equilibrium $(u_\infty,\Gamma_\infty)\in\cE_0$ in $\cSM_0$ as $t\to\infty$, then the properties $(i)$-$(iv)$ are valid.
\end{theorem}
\begin{proof}
The proof follows the same lines as that of Theorem \ref{Qual}.
\end{proof}

\bigskip

\noindent
{\bf Acknowledgment:}
J.P.\ and M.W.\ express their thanks for hospitality
to the Department of Mathematics at Vanderbilt University, where important parts of this work originated.
\goodbreak


\end{document}